\documentclass[12pt]{article}

\pdfoutput=1

\usepackage[a4paper]{geometry}
\usepackage{amsthm}
\usepackage{amsmath}
\usepackage{amsfonts}
\usepackage{amssymb}
\usepackage{graphicx}
\usepackage[colorlinks=true,citecolor=black,linkcolor=black,urlcolor=blue]{hyperref}
\newcommand{\doi}[1]{\href{http://dx.doi.org/#1}{\texttt{doi:#1}}}

\theoremstyle{plain}
\newtheorem{theorem}{Theorem}
\newtheorem{lemma}[theorem]{Lemma}
\newtheorem{corollary}[theorem]{Corollary}
\newtheorem{proposition}[theorem]{Proposition}

\theoremstyle{definition}
\newtheorem{definition}[theorem]{Definition}

\theoremstyle{remark}

\renewcommand{\emptyset}{\varnothing}

\begin{document}

\title{Isotropic matroids III\@: Connectivity}
\author{Robert Brijder\\Hasselt University\\Belgium\\{\small {\texttt{robert.brijder@uhasselt.be}}}
\and Lorenzo Traldi\\Lafayette College\\Easton, Pennsylvania 18042, USA\\{\small {\texttt{traldil@lafayette.edu}}}}
\date{}
\maketitle

\begin{abstract}
The isotropic matroid $M[IAS(G)]$ of a graph $G$ is a binary matroid, which is equivalent to the isotropic system introduced by Bouchet. In this paper we discuss four notions of connectivity related to isotropic matroids and isotropic systems. We show that the isotropic system connectivity defined by Bouchet is equivalent to vertical connectivity of $M[IAS(G)]$, and if $G$ has at least four vertices, then $M[IAS(G)]$ is vertically 5-connected if and only if $G$ is prime (in the sense of Cunningham's split decomposition). We also show that $M[IAS(G)]$ is $3$-connected if and only if $G$ is connected and has neither a pendant vertex nor a pair of twin vertices. Our most interesting theorem is that if $G$ has $n\geq7$ vertices then $M[IAS(G)]$ is not vertically $n$-connected. This abstract-seeming result is equivalent to the more concrete assertion that $G$ is locally equivalent to a graph with a vertex of degree $<\frac{n-1}{2}$.

\bigskip


\bigskip


\end{abstract}

\section{Introduction}

In this paper a \emph{graph} is a looped simple graph: each edge is incident on one or two vertices, and no two edges are incident on precisely the same vertices. An edge incident on just one vertex is a \emph{loop}. We use the terms \emph{adjacent} and \emph{neighbor} only in connection with non-loop edges; no vertex is its own neighbor, whether or not it is looped. We denote by $V(G)$ and $E(G)$ the set of vertices and edges of $G$, respectively. We use $n$ to denote $\left\vert V(G)\right\vert $ and $N_{G}(v)$ to denote the \emph{open neighborhood} $\{w\neq v\mid vw\in E(G)\}$.

We assume in this paper that the reader is familiar with basic notions of matroid theory, which can be found in \cite{O}. Also, in this paper the rows and columns of matrices are not ordered, but are instead indexed by some finite sets $X$ and $Y$, respectively. We refer to such a matrix as an $X \times Y$ matrix. The \emph{adjacency matrix} $A(G)$ of a graph $G$ is the $V(G)\times V(G)$ matrix over $GF(2)$ such that (1) if $v\in V(G)$, then the diagonal entry corresponding to $(v,v)$ is equal to $1$ if and only if $v$ is looped, and (2) if $v,w \in V(G)$ are distinct, then the entries corresponding to $(v,w)$ and $(w,v)$ are equal to $1$ if and only if $v$ and $w$ are neighbors. If $I$ denotes the $V(G)\times V(G)$ identity matrix, then the binary matroid represented by the matrix
\[
IAS(G)=%
\begin{pmatrix}
I & A(G) & A(G)+I
\end{pmatrix}
\]
is denoted $M[IAS(G)]$; we call $M[IAS(G)]$ the \emph{isotropic matroid} of $G$. Notice that our indexing convention assures us that for each $v \in V(G)$, the $v$ row of $IAS(G)$ is the concatenation of the $v$ rows of $I$, $A(G)$ and $A(G)+I$. The elements of $M[IAS(G)]$ corresponding to the $v$ columns of $I$, $A(G)$ and $A(G)+I$ are denoted $\phi(v)$, $\chi(v)$ and $\psi(v)$ respectively.

It is worth taking a moment to observe that the distinction between looped and unlooped vertices is more important in some parts of the theory than it is in others. For instance, looped vertices are quite useful in the discussion of~\cite{Tnewnew}, where isotropic matroids were introduced. In the present paper, though, we pay little attention to loops.

It was shown in~\cite{Tnewnew} that $M[IAS(G)]$ determines the delta-matroid and isotropic system associated with $G$. Delta-matroids and isotropic systems are combinatorial structures which have been studied by Andr\'{e} Bouchet and others for the last thirty years or so; we do not present detailed descriptions of them. 

A fundamental property of isotropic matroids (and isotropic systems) is that they detect a graph relation called local equivalence. That is, two graphs are locally equivalent if and only if their isotropic matroids (or isotropic systems) are isomorphic~\cite{Tnewnew}. Another fundamental property from~\cite{Tnewnew} is that isotropic matroids detect connectedness. That is, if $n \geq 2$ then $G$ is connected if and only if its isotropic matroid is connected if and only if its isotropic system is connected. (Connectedness of $M[IAS(G)]$ is the usual matroid idea; connectedness of isotropic systems is a more specialized notion.) The present paper concerns a question suggested by these properties: How is higher connectivity of $M[IAS(G)]$ reflected in the structure of $G$ and the graphs locally equivalent to $G$? 

The answer to our question is complicated by the fact that there are three different measures of connectivity of a matroid $M$: the (ordinary) connectivity $\tau(M)$, the cyclic connectivity $\kappa^{\ast}(M)$ and the vertical connectivity $\kappa(M)$. (We recall the definitions in Section~\ref{defs}.)

It turns out that nonempty isotropic matroids always have $\tau(M)=\kappa^{\ast}(M)$, and this value is determined in a simple way. Recall that if $v\neq w\in V(G)$, then $v$ and $w$ are \emph{twins} if they have $N_{G}(v)-\{w\}=N_{G}(w)-\{v\}$, and $v$ is \emph{pendant} on $w$ if $\{w\}=N_{G}(v)$.

\begin{theorem}
\label{cconnect}Every graph falls under precisely one of these four cases.

\begin{enumerate}

\item If $n=0$ then $\kappa^{\ast}(M[IAS(G)])=0$ and $\tau(M[IAS(G)])=\infty$.

\item If $n=1$ or $G$ is disconnected, then $\tau(M[IAS(G)])=\kappa^{\ast}(M[IAS(G)])=1$.

\item If $n>1$, $G$ is connected, and $G$ has a pendant vertex or a pair of twin vertices, then $\tau(M[IAS(G)])=\kappa^{\ast}(M[IAS(G)])=2$.

\item If $n>1$, $G$ is connected, and $G$ has neither a pendant vertex nor a pair of twin vertices, then $\tau(M[IAS(G)])=\kappa^{\ast}(M[IAS(G)])=3$.
\end{enumerate}
\end{theorem}

The inequality $\kappa(M)\geq\tau(M)$ holds for all nonempty isotropic matroids. The vertical connectivity is more interesting than the (cyclic) connectivity; it attains a broader range of values and is related to finer details of graph and matroid structure.

Recall the following definition of Cunningham~\cite{Cu}.

\begin{definition}
\label{prime}(\cite{Bconn,Cu}) A \emph{split} $(V_{1},W_{1};V_{2},W_{2})$ of a graph
$G$ is a partition $V(G)$ $=V_{1}\cup V_{2}$ such that $\left\vert V_{1}%
\right\vert ,\left\vert V_{2}\right\vert \geq2$, 
$W_{i}\subseteq V_{i}$ and every $v\in V_{1}$ has
\[
N_{G}(v)\cap V_{2}=
\begin{cases}
W_{2}\text{, if }v\in W_{1}\text{ }\\
\emptyset\text{, if }v\in V_{1}-W_{1}
\end{cases}
\text{.}%
\]
If $G$ has no split, then $G$ is said to be~\emph{prime}.
\end{definition}

Notice that according to Definition \ref{prime}, all graphs of order $\leq3$ are prime. This convention is not universal; some references explicitly require prime graphs to be of order $>3$. Notice also that Definition \ref{prime} can be restated using matrices: $G$ has a split with respect to $V_1$ and $V_{2}=V(G)-V_1$ if and only if  $\left\vert V_{1} \right\vert ,\left\vert V_{2}\right\vert \geq2$ and $r(A(G)[V_1,V_2]) \leq 1$, where $A(G)[V_1,V_2]$ is the submatrix of $A(G)$ involving rows from $V_1$ and columns from $V_2$, and $r$ denotes the rank over $GF(2)$.  

If $n\geq4$ then every pair of twin vertices $v$ and $w$
yields a split with $V_{1}=\{v,w\}$, and every pair consisting of a pendant
vertex $v$ and its neighbor $w$ yields a split with $V_{1}=\{v,w\}$. Moreover, if $G$ is a disconnected graph with a connected component $C$, then $G$ has a split with $V_{1}=V(C)$ if $\left\vert V(C)\right\vert, \left\vert V(G) - V(C)\right\vert \geq 2$; and  $G$ has a
split with $V_{1}=V(C)\cup\{x\}$ if $\left\vert V(G)-V(C)\right\vert >2$ and $x \in V(G) - V(C)$. Every connected 4-vertex graph has a pendant vertex or a pair of twins, so there is no prime
graph with $n=4$. Bouchet \cite{Bec} proved that every prime 5-vertex graph is
locally equivalent to the cycle graph $C_{5}$, and a special case of Bouchet's
obstructions theorem for circle graphs \cite{Bco} is that every prime 6-vertex
graph is locally equivalent to either $C_{6}$ or the wheel graph $W_{5}$.

\begin{theorem}
\label{vconnect}Every graph falls under precisely one of these five cases.

\begin{enumerate}

\item If $n\leq3$ and $G$ is connected, then $\kappa(M[IAS(G)])=n.$

\item If $G$ is disconnected then $\kappa(M[IAS(G)])=1$.

\item If $n\geq4$, $G$ is connected, and $G$ is not prime, then $\kappa(M[IAS(G)])=3$.

\item If $n\geq5$, $G$ is prime, and $\kappa(M[IAS(G)])<n$, then $\kappa(M[IAS(G)])$ is odd and $\geq 5$.

\item If $n\geq5$ and $\kappa(M[IAS(G)])=n$, then $G$ is locally equivalent to either the cycle graph $C_{5}$ or the wheel graph $W_{5}$.
\end{enumerate}
\end{theorem}

Theorems \ref{cconnect} and \ref{vconnect} indicate five significant differences between the (cyclic) connectivity and the vertical connectivity of an isotropic matroid. The first difference is that for $n\geq 1$, $\kappa^{\ast}(M[IAS(G)])$ and $\tau(M[IAS(G)])$ are always $\leq 3$ while $\kappa(M[IAS(G)])$ has a greater range of variation. (Some examples with $\kappa(M[IAS(G)])>5$ are discussed in Section~\ref{circle}.)

The second difference is that $\kappa(M[IAS(G)])$ is usually odd. (Indeed, up to isomorphism there are only three isotropic matroids whose vertical connectivity is even.) This odd parity is reflected in the fact that the vertical connectivity of $M[IAS(G)]$ is closely related to a notion of connectivity for isotropic systems introduced by Bouchet~\cite{Bconn}. We refer to the connectivity of the isotropic system with fundamental graph $G$ as the \emph{isotropic connectivity} of $G$, denoted $\kappa_{B}(G)$. The definition of $\kappa_{B}(G)$ is discussed in Section~\ref{vsep}, where we prove the following.

\begin{theorem}
\label{Bsep}If $G$ is a graph then the isotropic connectivity of $G$ and the vertical connectivity of $M[IAS(G)]$ are related as follows.
\[
\kappa_{B}(G)=
\begin{cases}
\frac{\kappa(M[IAS(G)])+1}{2} & \text{if }\kappa(M[IAS(G)])<n\\
\infty & \text{otherwise}
\end{cases}
\]
\end{theorem}

The third significant difference between Theorems \ref{cconnect} and \ref{vconnect} is that $\kappa^{\ast}(M[IAS(G)])$ and $\tau(M[IAS(G)])$ detect only pendant and twin vertices, while $\kappa(M[IAS(G)])$ detects arbitrary splits. The fact that split graphs with 5 or more vertices are singled out by the inequality $\kappa(M[IAS(G)]))\leq3$ may be deduced from Theorem \ref{Bsep} and Bouchet's result that split graphs are singled out by the inequality $\kappa_{B}(G)\leq2$ \cite[Theorem 11]{Bconn}. We provide the straightforward proofs of these results in Sections~\ref{vsep} and \ref{sec:char_max_connect}. 

The fourth significant difference between Theorems \ref{cconnect} and \ref{vconnect}, the singling out of $C_{5}$ and $W_{5}$ in case 5 of Theorem \ref{vconnect}, is more difficult to prove. The significance of case 5 is illuminated by the two following results.

\begin{theorem}
\label{degree}Let $G$ be a graph with $n\geq4$ vertices. Then $\kappa(M[IAS(G)])<n$ if and only if some graph locally equivalent to $G$ has a vertex of degree $<\frac{n-1}{2}$.
\end{theorem}

As $\left\vert V(C_{5})\right\vert =5$ and $\left\vert V(W_{5})\right\vert=6$, Theorems \ref{vconnect} and \ref{degree} imply a striking property of local equivalence.

\begin{corollary}
\label{lowdeg}Let $G$ be a graph with $n\geq7$ vertices. Then $G$ is locally equivalent to a graph with a vertex of degree $<\frac{n-1}{2}$.
\end{corollary}

Some relevant examples appear in Figure \ref{isomch3f1}. In the top row of the figure we see $C_{5}$, $W_{5}$, $W_{6}$ and $W_{7}$. Each graph in the bottom row is locally equivalent to the graph above it, and is a degree-minimal representative of its local equivalence class. That is, no graph locally equivalent to $C_{5}$ has any vertex of degree $\leq1$, no graph locally equivalent to $W_{5}$ or $W_{7}$ has any vertex of degree $\leq2$, and no graph locally equivalent to $W_{6}$ has either any vertex of degree $\leq1$ or more than three vertices of degree 2. (These assertions may all be verified by inspecting isotropic matroids, and using Corollary 7 of~\cite{BT1}.) Note that $C_{5}$ and $W_{5}$ witness the failure of Corollary \ref{lowdeg} for $n=5$ and $n=6$. Corollary \ref{lowdeg} also fails for $n\leq3$, of course, but it does hold for $n=4$.

\begin{figure}[bt]%
\centering
\includegraphics[scale=0.7]{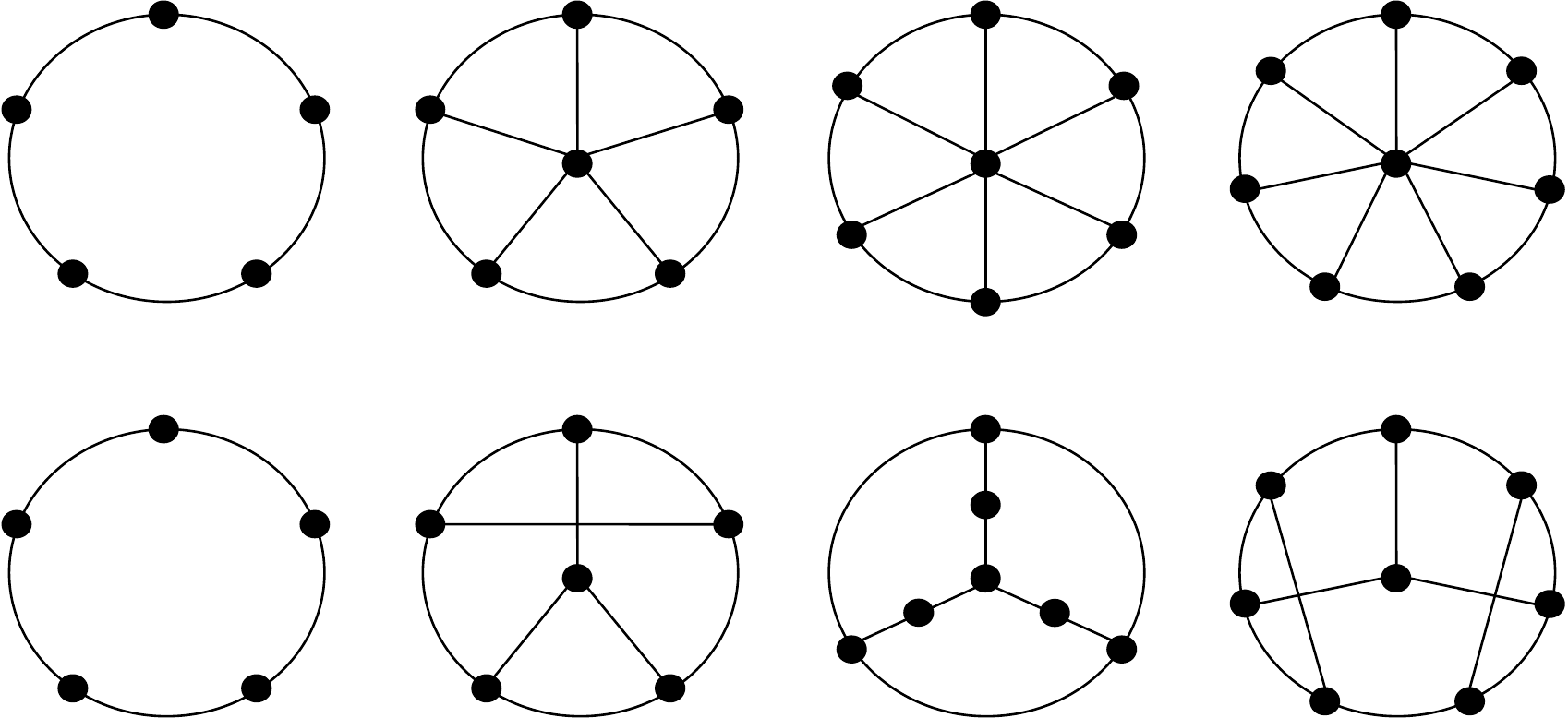}%
\caption{The graphs $C_{5}$, $W_{5}$, $W_{6}$ and $W_{7}$ above degree-minimal locally equivalent graphs.}%
\label{isomch3f1}%
\end{figure}

The fifth significant difference between Theorems \ref{cconnect} and \ref{vconnect} is that while the only local equivalence class uniquely determined by the (cyclic) connectivity of the corresponding isotropic matroid is the equivalence class of the empty graph, two nontrivial local equivalence classes are uniquely determined by the vertical connectivity.

\begin{corollary}
\label{unique} Let $G$ be a graph. Then $\kappa(M[IAS(G)])=2$ if and only if
$G$ is locally equivalent to $K_{2}$, and $\kappa(M[IAS(G)])=6$ if and only if $G$ is locally equivalent to $W_{5}$.
\end{corollary}

Most of the balance of the paper is devoted to proving the results already mentioned. In Section~\ref{sequent} we observe that there is a connection between splits of $G$ and the analysis of 3-separations of matroids due to Oxley, Semple and Whittle~\cite{OSW, OSW2}. In Section~\ref{circle} we deduce a characterization of circle graphs from the above results and Bouchet's circle graph obstructions theorem~\cite{Bco}.

\section{Definitions and notation}
\label{defs}

In this section we recall the definitions of local equivalence and matroid connectivity, and also the notation we use for isotropic matroids.

\begin{definition}
Let $G$ be a graph, with a vertex $v$.
\begin{itemize}
\item The \emph{loop complement} $G_{\ell}^{v}$ with respect to $v$ is the graph obtained from $G$ by reversing the loop status of $v$.
\item The \emph{simple local complement} $G_{s}^{v}$ with respect to $v$  is the graph obtained from $G$ by reversing all adjacencies between distinct elements of $N_{G}(v)$.
\item The \emph{non-simple local complement} $G_{ns}^{v}$ with respect to $v$ is the graph obtained from $G_{s}^{v}$ by performing loop complementations at all vertices in $N_{G}(v)$.
\end{itemize}
\end{definition}

\begin{definition}
\label{local}
Graphs $G$ and $H$ are \emph{locally equivalent} if one can be obtained from the other through some sequence of local complementations and loop complementations.
\end{definition}

We should mention that although the definition of local complementation for simple graphs is fairly standard, the looped version is not. For instance, loop complementations are not allowed in some references. Notice also that Definition~\ref{local} requires locally equivalent graphs to have the same vertex set. We sometimes abuse language with regard to this requirement, e.g., we might say ``$G$ is locally equivalent to $W_5$'' rather than ``$G$ is isomorphic to a graph locally equivalent to $W_5$.''

The definitions of the three types of matroid connectivity are rather complicated.

\begin{definition}
Let $M$ be a matroid with rank function $r=r_{M}$ and ground set $W$. Then $M$ has a \emph{connectivity function} $\lambda=\lambda_{M}:2^{W}\rightarrow \mathbb{N}\cup\{ 0 \}$, where for $S\subseteq W$, $\lambda(S)=r(S)+r(W-S)-r(M)$.
\end{definition}

\begin{definition}
\label{seps}
Let $M$ be a matroid with ground set $W$, and suppose $k$ is a positive integer. Then a subset $S\subseteq W$ is:

\begin{itemize}

\item a \emph{cyclic} $k$\emph{-separation} of $M$ if $\lambda(S)<k$ and both
$S$ and $W-S$ are dependent in $M$;

\item an \emph{ordinary} $k$\emph{-separation} of $M$ if $\lambda(S)<k$ and $\left\vert S\right\vert ,\left\vert W-S\right\vert \geq k$;

\item a \emph{vertical }$k$\emph{-separation} of $M$ if $\lambda(S)<k$ and
$r(S),r(W-S)\geq k$.
\end{itemize}
\end{definition}

An ordinary $k$-separation is simply called a $k$-separation in the literature; we use the adjective \emph{ordinary} to avoid confusion with the other two types of $k$-separations. This convention allows us to refer to a ``$k$-separation'' as any of the three types of $k$-separations. 

We mention four details regarding Definition~\ref{seps}. (\textit{i}) In all three cases, if $S$ is a $k$-separation then $k\geq \lambda(S)+1$ and $S$ is a $(\lambda(S)+1)$-separation. (\textit{ii}) If $S$ is a cyclic $k$-separation for any $k$, then $\lambda(S)<\left\vert S\right\vert + \left\vert W-S\right\vert-1-r(M)=\left\vert W\right\vert - 1-r(M)$, so $S$ is a cyclic $(\left\vert W\right\vert-r(M)-1)$-separation. (\textit{iii}) If $S$ is a vertical $k$-separation then $S$ is also an ordinary $k$-separation. (\textit{iv}) If $S$ is a vertical $k$-separation then $r(M)\geq r(S), r(W-S) \geq k > \lambda(S)$. The definition of $\lambda(S)$ then implies that $r(M)> r(S)$, so $r(M)>k$.

\begin{definition}
Let $M$ be a matroid with rank function $r$ and ground set $W$.

\begin{itemize}
\item We define $\tau(M)=\min(\{ k \mid M \text{ has an ordinary $k$-separation}\})$, where, by convention, $\min(\emptyset) = \infty$.\\ If $\tau(M)>1$ then $M$ is $j$-\emph{connected} for every $j \in \{2,\ldots,\tau(M)\}$.

\item We define $\kappa^*(M)=\min(\{ k \mid M \text{ has a cyclic $k$-separation}\} \cup \{\left\vert W\right\vert-r(M)\})$.\\ If $\kappa^{\ast}(M)>1$ then $M$ is \emph{cyclically} $j$-\emph{connected} for every $j \in \{2,\ldots,\kappa^{\ast}(M)\}$.

\item We define $\kappa(M)=\min(\{ k \mid M \text{ has a vertical $k$-separation}\} \cup \{r(M)\})$.\\ If $\kappa(M)>1$ then $M$ is \emph{vertically} $j$-\emph{connected} for every $j \in \{2,\ldots,\kappa(M)\}$.
\end{itemize}
\end{definition}

We refer to Oxley \cite{O} for a thorough introduction to the properties of the three types of matroid connectivity, but we take a moment to mention four more details. (\textit{v}) We follow Oxley's convention: if $M$ has no cyclic $k$-separation for any $k$ then $\kappa^{\ast}(M)=\left\vert W\right\vert-r(M)$, and if $M$ has no vertical $k$-separation for any $k$ then $\kappa(M)=r(M)$. Some other authors follow a different convention, according to which all three types of connectivity may equal $\infty$. (\textit{vi}) If $M$ is not 2-connected then $\tau(M)=1$; a component of $M$ is a 1-separation. In contrast, if $M$ is not cyclically or vertically 2-connected, then $\kappa^{\ast}(M)$ or $\kappa(M)$ may be $0$ or $1$. For instance, if $r(M)$ is $0$ or $1$ then $M$ has no vertical $k$-separation for any $k\geq1$, and $\kappa(M)=r(M)$. (\textit{vii}) The empty matroid has $\tau(\emptyset)=\infty$ but $\kappa^{\ast}(\emptyset)=\kappa(\emptyset)=0$. (\textit{viii}) The three notions of 2-connectedness differ in many cases. For instance, if $M$ is a vertically 2-connected matroid of rank $>1$ and $\ell$ is a loop then the direct sum $M \oplus \{\ell\}$ is not 2-connected; but it is vertically 2-connected.

We now recall the notation and terminology we use for isotropic matroids. If $G$ is a graph then the ground set of $M = M[IAS(G)]$ is denoted $W(G)$ or $W(M)$; it has $3n$ elements, one for each column of the matrix $IAS(G)$. If $v\in V(G)$, then the elements of $W(G)$ corresponding to the $v$ columns of $I$, $A(G)$ and $A(G)+I$ are denoted $\phi_{G}(v)$, $\chi_{G}(v)$ and $\psi_{G}(v)$ respectively. The set $\tau_{G}(v)=\{\phi_{G}(v)$, $\chi_{G}(v)$, $\psi_{G}(v)\}$ is the \emph{vertex triple} of $v$; observe that vertex triples are always dependent. (We rely on the subscript and the arguments to distinguish a vertex triple $\tau_{G}(v)$ from the connectivity $\tau(M)$ of a matroid.) If $X\subseteq V(G)$ then $\tau_{G}(X)=\cup_{x\in X}\tau_{G}(x)$. If a subset $S\subseteq W(G)$ does not intersect any vertex triple more than once, then $S$ is a \emph{subtransversal} of $W(G)$; if $S$ intersects every vertex triple precisely once then $S$ is a \emph{transversal}. In particular, a \emph{transverse matroid} of $G$ is a submatroid of $M$ obtained by restricting to a transversal, and a \emph{transverse circuit} of $G$ is a circuit of a transverse matroid. For each $v\in V(G)$ there is a special transverse circuit $\zeta_{G}(v)$, the \emph{neighborhood circuit} of $v$. It includes $\phi_{G}(w)$ for every $w\in N_{G}(v)$, and also includes either $\psi_{G}(v)$ (if $v$ is looped) or $\chi_{G}(v)$ (if $v$ is unlooped).

\section{Cyclic and ordinary connectivity of isotropic matroids}
\label{csep}

In the preceding section we mentioned that the properties ``2-connected,'' ``cyclically 2-connected'' and ``vertically 2-connected'' are different in general. For nonempty isotropic matroids, however, they are equivalent.

\begin{proposition}
\label{connected} If $n>0$ then an isotropic matroid $M=M[IAS(G)]$ is 2-connected, cyclically 2-connected or vertically 2-connected if and only if $G$ is connected and $n>1$. 
\end{proposition}

\begin{proof}
We first show the only-if direction. If $n=1$ then $M$ has three elements, a loop $\ell$ and a pair of parallel non-loops. The set $S=\{\ell\}$ is an ordinary 1-separation and a cyclic 1-separation, so $\tau(M)=\kappa^{\ast}(M)=1$. Also, $M$ has no vertical $k$-separation for any $k$, so $\kappa(M)=r(M)=1$.

If $G$ is not connected, let $C$ be a connected component of $G$ and let $S=\tau_{G}(V(C))$. Then no row of $IAS(G)$ has a nonzero entry in both a column corresponding to an element of $S$ and a column corresponding to an element of $W(G)-S$, so the rank of $IAS(G)$ is the sum of the ranks of the submatrices corresponding to $S$ and $W(G)-S$. That is, $\lambda(S)=0$. As $S$ and $W(G)-S$ are both dependent sets of cardinality $\geq3$ and rank $\geq1$, $S$ is an ordinary 1-separation, a cyclic 1-separation and a vertical 1-separation; consequently $\tau(M)=\kappa^{\ast}(M)=\kappa(M)=1$.

We conclude by contrapositive that if $M$ is 2-connected, cyclically 2-connected or vertically 2-connected then $G$ is connected and $n>1$.

For the converse, observe that if $e=vw$ is a nonloop edge of $G$ then (a) the vertex triples of $v$ and $w$ are circuits of $M$ and (b) the neighborhood circuit $\zeta_{G}(v)$ intersects both $\tau_{G}(v)$ and $\tau_{G}(w)$. Consequently a single component of $M$ contains $\tau_{G}(v)\cup\tau_{G}(w)$. This observation is true for every edge of $G$,
so if $G$ is a connected graph with $n>1$ then $M$ is a 2-connected matroid. Notice that a cyclic 1-separation is an ordinary 1-separation and a vertical 1-separation is also an ordinary 1-separation; as $M$ is 2-connected, we conclude that it
cannot have any kind of 1-separation. Finally, $n>1$ implies that $r(M)=n>1$ and $\left\vert W(G)\right\vert -r(M) = 3n-n = 2n > 1$, so $\kappa(M),\kappa^{\ast}(M)>1$ whether or not there is a cyclic or vertical $k$-separation for some $k>1$.
\end{proof}

We proceed to prove Theorem \ref{cconnect}. Suppose $n>1$ and $G$ is connected. If$\ v$ is any vertex of $G$
then the vertex triple $\tau_{G}(v)$ is a dependent set of $M[IAS(G)]$, whose complement is also dependent. As $r(\tau_{G}(v))\leq 2$, $\tau_{G}(v)$ is an ordinary 3-separation and a cyclic 3-separation. Consequently $\tau(M[IAS(G)])\leq3$
and $\kappa^{\ast}(M[IAS(G)])\leq3$.

Theorem \ref{cconnect} now follows from Proposition~\ref{connected} and this:

\begin{proposition}
Suppose $G$ is connected and $n>1$. Then the following are equivalent.

\begin{enumerate}
\item $M[IAS(G)]$ is not 3-connected.

\item $M[IAS(G)]$ is not cyclically 3-connected.

\item $M[IAS(G)]$ has a circuit of size 2.

\item $G$ has a pendant vertex, or a pair of twin vertices.
\end{enumerate}
\end{proposition}

\begin{proof}
We begin with a simple observation. Given a subset $S\subseteq W(G)$, let
$S_{\phi}, S_{\chi}, S_{\psi}\subseteq V(G)$ be the subsets with 
$ S=\phi_{G}(S_{\phi})\cup \chi_{G}(S_{\chi}) \cup \psi_{G}(S_{\psi})$.
Then the rank of $S$ in $M=M[IAS(G)]$ is the $GF(2)$-rank of a matrix%
\[
\bordermatrix{
& \phi_{G}(S_{\phi}) & \chi_{G}(S_{\chi}) \cup \psi_{G}(S_{\psi}) \cr
S_{\phi} & I & *  \cr
V(G)- S_{\phi} & 0 & B \cr
}\text{,}%
\]
where $I$ is the $S_{\phi} \times S_{\phi}$ identity matrix and the off-diagonal entries of $B$ record adjacencies between vertices in $S_{\chi} \cup S_{\psi}$ and vertices outside $S_{\phi}$. Observe that $r(S)\geq\left\vert S_{\phi} \right\vert $, and if $r(S)=\left\vert S_{\phi} \right\vert $ then no element of $S_{\chi} \cup S_{\psi}$ has a neighbor in $V(G)-S_{\phi}$.

If condition 4 holds then $IAS(G)$ has a pair of identical columns. This implies that $M$ has a pair of parallel elements, so condition 3 holds. Another simple argument shows that condition 3 implies conditions 1 and 2: If
$S$ is a circuit of size 2 in $M$ then as the sum of the columns of $IAS(G)$ is $0$, $W(G)-S$ is a dependent set. As $\left\vert S\right\vert $, $\left\vert W(G)-S\right\vert \geq2$ it follows that $S$ is an ordinary 2-separation and a cyclic 2-separation.

The proof is completed by showing that either of conditions 1, 2 implies condition 4.

Suppose condition 1 or condition 2 holds, and let $S$ be an ordinary or cyclic 2-separation of $M$. Proposition \ref{connected} tells us that $S$ cannot be a 1-separation of either type, so it must be that $r(S)+r(W(G)-S)=n+1$. The observation of the first paragraph tells us that after interchanging $S$ and $W(G)-S$ if necessary, we may presume that $r(S)=\left\vert S_{\phi} \right\vert $ and $r(W(G)-S)=\left\vert (W(G) - S)_{\phi} \right\vert +1$. Consequently $S$ is spanned by $\phi_{G}(S_{\phi})$, and $W(G)-S$ is spanned by $\phi_{G}(V(G)-S_{\phi})$ together with one additional element.

If $S_{\phi}=V(G)$ then $r(W(G)-S)=1$. Choose any two elements of $W(G)-S$; they must be parallel, so the corresponding columns of $IAS(G)$ must be the same. Considering the definition of $IAS(G)$, it is easy to see that if two non-$\phi$ columns are the same then the corresponding vertices of $G$ are twins.

If $S_{\phi}$ is a proper subset of $V(G)$ then as $G$ is connected, there is an edge $vw$ with $v\in S_{\phi}$ and $w\notin S_{\phi}$. For each such edge, the $\chi_{G}(v)$ and $\psi_{G}(v)$ columns of $IAS(G)$ have nonzero coordinates with respect to $w$, so these elements are not contained in the span of $\phi_{G}(S_{\phi})$. Consequently $\chi_{G}(v), \psi_{G}(v) \in W(G)-S$. It follows that the span of $W(G)-S$ also includes $\phi_{G}(v)$; as noted at the end of the paragraph before last, this implies that $W(G)-S$ is contained in the span of $\phi_{G}(\{v\} \cup (V(G)-S_{\phi}))$. It follows that no column of $IAS(G)$ that corresponds to an element of $W(G)-S$ has a nonzero $x$ coordinate for any vertex $x\notin\{v\}\cup (V(G)-S_{\phi})$. In particular, the $\chi_{G}(v)$ and $\psi_{G}(v)$ columns of $IAS(G)$ do not have a nonzero $x$ coordinate for any vertex $x\notin\{v\}\cup (V(G)-S_{\phi})$, so the neighbors of $v$ all lie in $V(G)-S_{\phi}$. 

If $x\notin\{v\}\cup (V(G)-S_{\phi})$ then either the $\chi_{G}(x)$ or the $\psi_{G}(x)$ column of $IAS(G)$ has a nonzero $x$ coordinate, so $\chi_{G}(x)$ or $\psi_{G}(x)$ must be an element of $S$. As $S$ is spanned by $\phi_{G}(S_{\phi})$, it follows that the neighbors of $x$ all lie in $S_{\phi}$. As the neighbors of $v$ all lie in $V(G)-S_{\phi}$, it follows that no edge of $G$ connects any $x\notin\{v\}\cup (V(G)-S_{\phi})$ to any $y\in \{v\}\cup (V(G)-S_{\phi})$. Now, $G$ is connected, so we conclude that there is no vertex $x\notin\{v\}\cup (V(G)-S_{\phi})$. That is, $\{v\}=S_{\phi}$. As $S$ is an ordinary or cyclic 2-separation, $S$ must contain some element other than $\phi_{G}(v)$; this other element must be $\chi_{G}(w)$ or $\psi_{G}(w)$ for some $w\neq v$. The fact that $S$ is spanned by $\phi_{G}(S_{\phi})=\{\phi_{G}(v) \}$ implies that $w$ has no neighbor other than $v$.

We conclude that if condition 1 holds or condition 2 holds, then condition 4 holds.
\end{proof}

\section{Isotropic separations and vertical separations}
\label{vsep}

In this section we discuss vertical separations of isotropic matroids, and prove Theorem~\ref{Bsep}.

For an $n$-vertex graph $G$ with adjacency matrix $A=A(G)$, Bouchet~\cite{Bconn} defines the \emph{cut-rank} function $c_{G}:2^{V(G)}\rightarrow \mathbb{N} \cup \{0\}$, where for all $X\subseteq V(G)$, $c_{G}(X)=r(A[V(G)-X,X])$, the $GF(2)$-rank of the submatrix of $A$ that includes the columns from $X$ and the rows from $V(G)-X$. (The rank of the empty matrix is taken to be $0$.) 

Notice that
\[
c_{G}(X)=c_{G}(V(G)-X)\leq\min\{\left\vert X\right\vert ,\left\vert
V(G)-X\right\vert \}
\leq
\left\lfloor \frac{n}{2}\right\rfloor \text{ }\forall
X\subseteq V(G).
\]
The equality $c_{G}(X)=0$ holds if $X$ or $V(G)-X$ is a union of vertex-sets of connected components of $G$. The equality $c_{G}(X)=1$ holds if $X$ or $V(G)-X$ consists of a single non-isolated vertex, and also if $G$ has a split $(V_{1},W_{1};V_{2},W_{2})$ with $X=V_{1}$ and $W_1 \neq \emptyset \neq W_2$.

It turns out that we can define $c_{G}$ using the isotropic matroid $M=M[IAS(G)]$. Define $c_M: 2^{V(G)} \to \mathbb{N} \cup \{0\}$, where for all $X\subseteq V(G)$, $c_M(X) = r(\tau_G(X)) - |X|$.

\begin{lemma}
\label{lem:trans_cut_rank} Let $G$ be a graph and $M=M[IAS(G)]$. Then $c_G = c_M$. In other words, for every $X\subseteq V(G)$, $r(\tau_{G}(X))=\left\vert X\right\vert +c_{G}(X)$. Moreover, $\lambda(\tau_{G}(X))=2c_{G}(X)$.
\end{lemma}

\begin{proof}
Let $A = A(G)$. We have that $r(\tau_{G}(X))$ is the $GF(2)$-rank of the matrix%
\[
\bordermatrix{
& \phi_{G}(X) & \chi_{G}(X) & \psi_{G}(X) \cr
X & I & A[X,X] & A[X,X]+I \cr
V(G)-X & 0 & A[V(G)-X,X] & A[V(G)-X,X] \cr
}.
\]
Thus $r(\tau_{G}(X))=r(I)+r(A[V(G)-X,X])=|X|+c_{G}(X)$. It follows that%
\begin{align*}
\lambda(\tau_{G}(X))  &  =r(\tau_{G}(X))+r(\tau_{G}(V(G)-X))-r(M)\\
&  =\left\vert X\right\vert +c_{G}(X)+\left\vert V(G)-X\right\vert
+c_{G}(V(G)-X)-n=2c_{G}(X)\text{.}%
\end{align*}
\end{proof}

Since $c_G = c_M$ and local equivalences of graphs induce isomorphisms of isotropic matroids that preserve vertex triples~\cite[Section 3]{BT1}, we obtain the well-known result of Bouchet~\cite{Bconn} that $c_G = c_{G'}$ when $G$ and $G'$ are locally equivalent.

\begin{proposition}
\label{purechar}Let $X\subseteq V(G)$. Then $c_{G}(X)<\min\{\left\vert X\right\vert ,\left\vert V(G)-X\right\vert \}$ if and only if there is a $k\geq1$ such that $\tau_{G}(X)$ is a vertical $k$-separation of $M=M[IAS(G)]$. If this is the case then the smallest $k$ for which $\tau_{G}(X)$ is a vertical $k$-separation of $M$ is $k=2c_{G}(X)+1$.
\end{proposition}

\begin{proof}
Suppose $c_{G}(X)<\min\{\left\vert X\right\vert ,\left\vert V(G)-X\right\vert \}$. Then $r(\tau_{G}(X))=\left\vert X\right\vert +c_{G}(X)\geq 2c_{G}(X)+1$, 
$r(\tau_{G}(V(G)-X))=\left\vert V(G)-X\right\vert +c_{G}(X)\geq2c_{G}(X)+1$ and $\lambda(\tau_{G}(X))=2c_{G}(X)<2c_{G}(X)+1$, so $\tau_{G}(X)$ is a vertical $(2c_{G}(X)+1)$-separation of $M$.

For the converse, suppose $\tau_{G}(X)$ is a vertical $k$-separation of $M$. Then $k\geq \lambda(\tau_{G}(X))+1=2c_{G}(X)+1$. Also, $k\leq r(\tau_{G}(X))=\left\vert X\right\vert +c_{G}(X)$ and $k\leq r(\tau_{G}(V(G)-X))=\left\vert V(G)-X\right\vert +c_{G}(X)$, so $2c_{G}(X)<\left\vert X\right\vert +c_{G}(X)$ and $2c_{G}(X)<\left\vert V(G)-X\right\vert +c_{G}(X)$. It follows that $c_{G}(X)<\min\{\left\vert X\right\vert ,\left\vert V(G)-X\right\vert \}$.
\end{proof}

\begin{proposition}
\label{puresep} Suppose $G$ is a graph, and $M=M[IAS(G)]$ is such that $\kappa(M)<n$. Then $M$ has a vertical $\kappa(M)$-separation $\tau_{G}(X)$ for some $X \subseteq V(G)$, and $\kappa(M)=2c_{G}(X)+1$.
\end{proposition}
\begin{proof}
Let $k = \kappa(M)$. Since $k < n = r(M)$, $M$ has a vertical $k$-separation but no vertical $k^{\prime}$-separation with $k^{\prime}<k$.

If $k=1$ then $M$ is not 2-connected and $r(M) = n > k = 1$, so $G$ is disconnected and $n>1$. If $C$ is a connected component of $G$ then $X=\tau_{G}(V(C))$ is a vertical 1-separation of $M$. Every entry of $A[V(C),V(G)-V(C)]$ is $0$, so $c_{G}(X)=0$ and $k=1=2c_{G}(X)+1$.

Suppose $k>1$, and let $S$ be a vertical $k$-separation of $M$ that is not of the form $\tau_{G}(X)$. Then
interchanging the names of $S$ and $W(G)-S$ if necessary, we may presume that
there is a vertex triple $\tau_{G}(v)$ such that $S$ contains precisely two
elements of $\tau_{G}(v)$. Let $S^{\prime}=S\cup\tau_{G}(v)$. Then
$r(S^{\prime})=r(S)$ and $r(W(G)-S^{\prime})\in\{r(W(G)-S)-1,r(W(G)-S)\}$. If
$r(W(G)-S^{\prime})=r(W(G)-S)-1$ then we have $r(S^{\prime}),r(W(G)-S^{\prime
})\geq k-1$ and $r(S^{\prime})+r(W(G)-S^{\prime}%
)=r(S)+r(W(G)-S)-1<r(W(G))+k-1$, so $S^{\prime}$ is a vertical $(k-1)$-separation, contradicting the choice of $k$. Consequently $r(W(G)-S^{\prime})=r(W(G)-S)$, so $S^{\prime}$ is a vertical $k$-separation. Repeating this modification as many times as necessary, we ultimately obtain a vertical
$k$-separation $\tau_{G}(X)$. Proposition \ref{purechar} tells us that $k=2c_{G}(X)+1$. 
\end{proof}

\begin{corollary}
\label{vconn}
If every $X \subseteq V(G)$ has $c_{G}(X)=\min\{\left\vert X\right\vert ,\left\vert V(G)-X\right\vert \}$, then the vertical connectivity of $M=M[IAS(G)]$ is $n$. Otherwise, 
\[\kappa(M)=1+2\cdot\min\{c_{G}(X)\mid X\subseteq V(G)\text{ has }c_{G}(X)<\min\{\left\vert X\right\vert ,\left\vert V(G)-X\right\vert\}\}.
\]
\end{corollary}
\begin{proof}
Suppose $M$ has a vertical $k$-separation for some $k$. Then $M$ has a vertical $\kappa(M)$-separation but no vertical $k^{\prime}$-separation with $k^{\prime}<\kappa(M)$, so Proposition \ref{puresep} tells us that there is a subset $X_{0}\subseteq V(G)$ such that $\kappa(M)=1+2c_{G}(X_{0})<n$, and Proposition~\ref{purechar} tells us that $c_{G}(X_0)<\min\{\left\vert X_{0} \right\vert ,\left\vert V(G)-X_{0} \right\vert \}$. On the other hand, Proposition \ref{purechar} implies that every subset $X\subseteq V(G)$ has $\kappa(M)\leq1+2c_{G}(X)$. For if $\tau_{G}(X)$ is not a vertical $k$-separation for any $k$, then $1+2c_{G}(X)=n>\kappa(M)$; and if $\tau_{G}(X)$ is a vertical $k$-separation for some $k$, then $\tau_{G}(X)$ is a vertical $(2c_{G}(X)+1)$-separation, so $2c_{G}(X)+1\geq\kappa(M)$.

If $M$ has no vertical $k$-separation for any $k$ then
$\kappa(M)=n$ by definition, and Proposition \ref{purechar} tells us that $c_{G}(X)=\min\{\left\vert X\right\vert ,\left\vert V(G)-X\right\vert \}$ for all $X \subseteq V(G)$.
\end{proof}

Corollary~\ref{vconn} tells us that $\kappa(M)$ cannot be even unless $n$ is even and $M$ has no vertical $k$-separation for any $k$.

Bouchet \cite{Bconn} described the connectivity of an isotropic system using the cut-rank function of any graph associated with that isotropic system. This concept was studied further by Allys \cite{A} and Bouchet and Ghier \cite{BG}.

\begin{definition}
\label{isosep}(\cite{Bconn}) Let $G$ be a graph. Then a subset $X\subseteq V(G)$ is an \emph{isotropic $k$-separation} of $G$ if $\left\vert X\right\vert ,\left\vert V(G)-X\right\vert \geq k$ and $c_{G}(X) < k$.
\end{definition}

\begin{definition}
\label{isoconnect}(\cite{Bconn}) Let $G$ be a graph. Then the \emph{isotropic connectivity} of $G$ is defined as
$$
\kappa_{B}(G)= \min(\{k \mid \text{there is an isotropic $k$-separation of $G$}\}),
$$
where, again, by convention, we take $\min(\emptyset) = \infty$.
\end{definition}

Notice that if $X$ satisfies Definition \ref{isosep} for some value of $k$, then we have $c_{G}(X)<\min\{\left\vert X\right\vert ,\left\vert V(G)-X\right\vert \}$ and the smallest value of $k$ for which $X$ satisfies Definition \ref{isosep} is $c_{G}(X)+1$. We deduce Theorem~\ref{Bsep}:
\[
\kappa_{B}(G) =
\begin{cases}
\infty & \text{if }\kappa(M[IAS(G)])=n\\
\frac{\kappa(M[IAS(G)])+1}{2} & \text{otherwise}
\end{cases}.
\]

\section{Splits of \texorpdfstring{$G$}{\textit{G}} and separations of \texorpdfstring{$M[IAS(G)]$}{\textit{M[IAS(G)]}}}
\label{sequent}

Before providing proofs of Theorems~\ref{vconnect} and~\ref{degree}, we take a moment to point out an interesting consequence of the discussion above. Namely: if $G$ is a graph then there is a close relationship between splits of $G$ and decompositions of the matroid $M[IAS(G)]$.

Theorem~\ref{cconnect} gives us the easy parts of this relationship. First, if $G$ is not connected then components of $G$ and $M[IAS(G)]$ correspond, with the proviso that an isolated vertex of $G$ corresponds to two matroid components. Second, if $G$ is connected then two-element splits of $G$ correspond to ordinary 2-separations of $M[IAS(G)]$. These in turn give rise to a description of $M[IAS(G)]$ using 2-sums; see~\cite[Section 8.3]{O} for details regarding the relation between 2-sums and ordinary 2-separations.

Suppose now that $M[IAS(G)]$ is 3-connected. Then Theorem~\ref{cconnect} tells us that $G$ is connected and $n\geq5$, as every connected graph with $n\leq4$ has a pendant vertex or a pair of twins. It turns out that splits of $G$ are related to a certain type of ordinary 3-separation, which has been studied by Oxley, Semple and Whittle~\cite{OSW, OSW2}. We begin with a couple of lemmas.

\begin{lemma}
\label{bigrank}
Let $M$ be a 3-connected isotropic matroid.
\begin{enumerate}
\item Every subset $S\subseteq W(M)$ with $\left\vert S \right\vert \geq 4$ has $r(S)\geq3$.
\item Every subset $S\subseteq W(M)$ with $\left\vert S \right\vert \geq 3n-5$ has $r(S)=n$.
\end{enumerate}
\end{lemma}

\begin{proof}
As $M$ is 3-connected, it has no loop or pair of parallels. The only 4-element matroid of rank 2 without a loop or parallels is $U_{2,4}$, which is not binary. Consequently the rank of a 4-element set in $M$ must be $\geq3$.

The second assertion takes a little more work. Let $M = M[IAS(G)]$. Suppose first that there is a subset $X\subseteq V(G)$ with $\left\vert X \right\vert = n-2$ and $\tau_{G}(X)\subseteq S$. If $c_{G}(X)=0$ then no edge of $G$ connects $X$ to $V(G)-X$, so $G$ is not connected. If $c_{G}(X)=1$ then $G$ has a split $(V_1,W_1;V_2,W_2)$ with $V_1=X$; either the two elements of $V(G)-X$ are twins, or one is pendant on the other. According to Theorem~\ref{cconnect}, then, the hypothesis that $M$ is 3-connected requires $c_{G}(X)\geq 2$; as $c_{G}(X)\leq \left\vert V(G)-X \right\vert =2$, it follows that $c_{G}(X) = 2$. Lemma~\ref{lem:trans_cut_rank} tells us that $r(\tau_{G}(X))=n-2+c_{G}(X)=n$, so $r(S)=n$ too.

Now, suppose there is no subset $X\subseteq V(G)$ with $\left\vert X \right\vert = n-2$ and $\tau_{G}(X) \subseteq S$. Then there is a vertex $v\in V(G)$ such that $S$ includes precisely two elements of $\tau_{G}(v)$. Let $\{v_1,\ldots,v_j\}$ include all such vertices, and let $S^{\prime}=S\cup \tau_{G}(v_1) \cup \cdots \cup \tau_{G}(v_j)$. Each vertex triple $\tau_{G}(v_i)$ is dependent, so $r(S)=r(S^{\prime})$. The argument of the preceding paragraph applies to $S^{\prime}$, so $r(S^{\prime})=n$.
\end{proof}

\begin{corollary}
\label{foursep}
Let $M$ be a 3-connected isotropic matroid with ground set $W$. Then
\begin{eqnarray*}
&&\{\text{ordinary 3-separations S of }M \text{ such that } \left\vert S \right\vert,\left\vert W-S \right\vert \geq 6\} \\
&=&\{\text{ordinary 3-separations S of }M \text{ such that } \left\vert S \right\vert,\left\vert W-S \right\vert \geq 4\} \\
&=& \{\text{vertical 3-separations of }M\} \text{.}
\end{eqnarray*}
\end{corollary}

\begin{proof}
It is obvious that the first set is contained in the second.

If $S$ is an element of the second set, then the first assertion of Lemma~\ref{bigrank} tells us that $r(S),r(W-S) \geq 3$. It follows that $S$ is a vertical 3-separation of $M$. 

Now, let $S$ be a vertical 3-separation of $M$. Without loss of generality, we may presume that $\left\vert S \right\vert \geq \left\vert W-S \right\vert$. Properties (\textit{iii}) and (\textit{iv}) of Section~\ref{defs} tell us that $S$ is an ordinary 3-separation of $M$, and $r(M)>r(S)$. The second assertion of Lemma~\ref{bigrank} then tells us that $\left\vert S \right\vert \leq 3n-6=\left\vert W \right\vert-6$, so $\left\vert S \right\vert \geq \left\vert W-S \right\vert \geq 6$.
\end{proof}

\begin{corollary}
\label{bigsep}
A 3-connected isotropic matroid has no ordinary or vertical 3-separation of size 4 or 5.
\end{corollary}

The theory of Oxley, Semple and Whittle~\cite{OSW, OSW2} excludes a special type of ordinary 3-separation.

\begin{definition}
\label{sequential}
Let $S$ be an ordinary 3-separation of a matroid $M$. Then $S$ is \emph{sequential} if the elements of $S$ can be ordered as $s_1,\ldots,s_m$ in such a way that $\lambda(\{s_1,\ldots,s_i\})<3$ for every $i\in\{1,\ldots,m\}$. $S$ is also considered to be sequential if its complement has such an ordering. 
\end{definition}

\begin{corollary}
\label{3sequential}
Let $S$ be an ordinary 3-separation of a 3-connected isotropic matroid. Then $S$ is sequential if and only if either $S$ or its complement is of cardinality 3.
\end{corollary}

\begin{proof}
If $S$ or its complement is of cardinality 3, then every ordering of the three elements will satisfy Definition~\ref{sequential}. 

For the converse, suppose instead that $S=\{s_1,\ldots,s_m\}$ is a sequential ordinary 3-separation of $M=M[IAS(G)]$ with $m>3$, and that the given order satisfies Definition~\ref{sequential}. Then $S^{\prime}=\{s_1,s_2,s_3,s_4\}$ is an ordinary 3-separation of $M$, contradicting Corollary~\ref{bigsep}.
\end{proof}

Corollaries~\ref{foursep} and~\ref{3sequential} tell us that if $G$ is a connected graph with no pendant or twin vertices, then the non-sequential ordinary 3-separations of $M[IAS(G)]$ and the vertical 3-separations of $M[IAS(G)]$ coincide.

The theory of non-sequential ordinary 3-separations presented by Oxley, Semple and Whittle~\cite{OSW, OSW2} includes the following notion of equivalence.

\begin{definition}
If $M$ is a matroid with ground set $W$ and $S\subseteq W$, then the \emph{full closure} of $S$ is the smallest subset $\operatorname{fcl}(S)\subseteq W$ that contains $S$ and is closed in both $M$ and $M^{\ast}$. Two ordinary 3-separations $S$ and $S^{\prime}$ are \emph{equivalent} if the sets $\{\operatorname{fcl}(S),\operatorname{fcl}(W-S)\}$ and $\{\operatorname{fcl}(S^{\prime}),\operatorname{fcl}(W-S^{\prime})\}$ are the same.
\end{definition}

\begin{proposition}
Suppose $M=M[IAS(G)]$ is 3-connected. Then every non-sequential ordinary 3-separation of $M$ is equivalent to an ordinary 3-separation $\tau_{G}(X)$, where $G$ has a split $(V_1,W_1;V_2,W_2)$ with $V_1=X$.
\end{proposition}

\begin{proof}
Let $S$ be a non-sequential ordinary 3-separation which is not of the form $\tau_{G}(X)$. Interchanging the labels of $S$ and $W(G)-S$ if necessary, we may presume that there is a vertex $v$ such that $\left\vert \tau_{G}(v) \cap S \right\vert=2$. 

Let $S^{\prime}=S\cup \tau_{G}(v)$. Then $r(S)=r(S^{\prime})$, and $r(W(G)-S^{\prime})\leq r(W(G)-S)$. If $r(W(G)-S^{\prime})< r(W(G)-S)$ then $S^{\prime}$ is an ordinary 2-separation, an impossibility as $M$ is 3-connected; hence $r(W(G)-S^{\prime})= r(W(G)-S)$. It follows that $S^{\prime}$ is also an ordinary 3-separation. Corollaries~\ref{foursep} and~\ref{3sequential} tell us that $\left\vert S \right\vert, \left\vert W(G)-S \right\vert \geq 6$, so it must be that $\left\vert S^{\prime} \right\vert, \left\vert W(G)-S^{\prime} \right\vert \geq 5$; then Corollary~\ref{3sequential} tells us that $S^{\prime}$ is a non-sequential 3-separation. Moreover $S$ and $S^{\prime}$ have the same closure, and $W(G)-S$ and $W(G)-S^{\prime}$ have the same closure. 

Repeating this process as many times as possible, we eventually obtain a non-sequential ordinary 3-separation $\tau_{G}(X)$ such that $S$ and $\tau_{G}(X)$ have the same closure, and $W(G)-S$ and $W(G)-\tau_{G}(X)$ have the same closure. Then $S$ is equivalent to $\tau_{G}(X)$.

As $\tau_{G}(X)$ is a non-sequential ordinary 3-separation, Corollary~\ref{3sequential} tells us that $\left\vert \tau_{G}(X) \right\vert$ and $\left\vert W(G)-\tau_{G}(X) \right\vert$ are both $>3$. Consequently $\left\vert X \right\vert$ and $\left\vert V(G)-X \right\vert$ are both $>1$. Also, as noted after Corollary~\ref{3sequential} the fact that $\tau_{G}(X)$ is a non-sequential ordinary 3-separation implies that $\tau_{G}(X)$ is also a vertical 3-separation. Then Proposition~\ref{purechar} implies that $c_G(X)=1$, so $G$ has a split $(V_1,W_1;V_2,W_2)$ with $V_1=X$.
\end{proof}

As discussed in~\cite[Section 9.3]{O}, it follows that if $M[IAS(G)]$ is 3-connected then the splits in $G$ give rise to a description of $M[IAS(G)]$ using 3-sums.

\section{Theorem \ref{degree}}
\label{degreeproof}

In this section we prove an expanded form of Theorem \ref{degree}. We begin with a result about small transverse circuits.

\begin{proposition}
\label{smallcirc}Let $q$ be the cardinality of the smallest transverse circuit(s) in $M=M[IAS(G)]$. If $q<\frac{n+1}{2}$, then $2q-1\geq\kappa(M)$.
\end{proposition}

\begin{proof}
Let $A=A(G)$. If $\zeta$ is a transverse circuit of size $q\leq\frac{n}{2}$ and $X=\{x\in V(G)\mid\tau_{G}(x)\cap\zeta\neq\emptyset\}$, then the columns of $IAS(G)$ corresponding to the elements of $\zeta =  \zeta \cap \tau_{G}(X)$ sum to $0$. It follows that if $X_{0}=\{x\in
X\mid \chi_{G}(x) \in \zeta$ or $\psi_{G}(x) \in \zeta\}$ then the columns of $A[V(G)-X,X]$ corresponding to elements of $X_{0}$ sum to $0$. Consequently 
\[
c_{G}(X)=r(A[X,V(G)-X])<q=\left\vert X\right\vert \leq\left\vert
V(G)-X\right\vert .
\]
Then Corollary \ref{vconn} tells us that $\kappa(M)\leq1+2c_{G}(X)\leq1+2(q-1)$.
\end{proof}

The inequality of Proposition \ref{smallcirc} is strict for some graphs. For instance, if $G$ is the interlacement graph of an Euler circuit of $K_{4,4}$ then $\kappa(M[IAS(G)])=5$, but the smallest transverse circuits are of size 4.

\begin{theorem}
\label{expdegree}If $n\geq4$ then these conditions are equivalent.

\begin{enumerate}
\item Some graph locally equivalent to $G$ has a vertex of degree $<\frac{n-1}{2}$.

\item $G$ has a transverse circuit of size $<\frac{n+1}{2}$.

\item There is a subset $X\subseteq V(G)$ such that $\tau_{G}(X)$ and $\tau_{G}(V(G)-X)$ both contain transverse circuits.

\item $\kappa(M[IAS(G)])<n$.
\end{enumerate}
\end{theorem}

\begin{proof}
The equivalence of conditions 1 and 2 follows from this obvious consequence of the theory developed in \cite{BT1}:

\begin{proposition}
Let $G$ be a graph, and let $q$ be the size of the smallest transverse circuit(s) in $M=M[IAS(G)]$. Then $q-1$ is the smallest degree of any vertex in any graph locally equivalent to $G$.
\end{proposition}

Proposition \ref{smallcirc} tells us that condition 2 implies condition 4, and condition 3 immediately implies condition 2, so it suffices to show that condition 4 implies condition 3. Recall that Corollary \ref{vconn} tells us that if $\kappa(M)<n$ then there is a subset $X\subseteq V(G)$ such that $c_{G}(X)<\min\{\left\vert X\right\vert ,\left\vert V(G)-X\right\vert \}$. Then the columns of the matrix $A[V(G)-X,X]$ are linearly dependent, so there is some subset $X_{0}\subseteq X$ such that the columns of $A[V(G)-X,X]$ corresponding to the elements of $X_{0}$ sum to $0$. Equivalently, if $S=\{\chi_{G}(x)\mid x\in X_{0}\}$ then the sum of the columns of $IAS(G)$ corresponding to elements of $S$ is a column vector $\Sigma$ whose $v$ coordinate is 0\ for every $v\in V(G)-X$. Let $S^{\prime}\subseteq\tau_{G}(X)$ be the set obtained from $S$ as follows. First: for every $x\in X-X_{0}$ such that the $x$ coordinate of $\Sigma$ is 1, insert $\phi_{G}(X)$. Second: for every $x\in X_{0}$ such that the $x$ coordinate of $\Sigma$ is 1, remove $\chi_{G}(x)$ and replace it with $\psi_{G}(X)$. The effect of these changes is to add 1 to every nonzero coordinate of $\Sigma$, so the sum of the columns of $IAS(G)$ corresponding to elements of $S^{\prime}$ is $0$. As $S^{\prime}$ is a subtransversal contained in $\tau_{G}(X)$, we conclude that $\tau_{G}(X)$ contains a transverse circuit. Interchanging $X$ and $V(G)-X$, the same argument tells us that $\tau_{G}(V(G)-X)$ also contains a transverse circuit.
\end{proof}

\begin{corollary}
If $\kappa(M[IAS(G)])=n$ then the diameter of $G$ is $\leq2$.
\end{corollary}

\begin{proof}
Suppose instead that the diameter of $G$ is greater than 2. Then there are
vertices $v\neq w\in V(G)$ such that $X=N_{G}(v)\cup\{v\}$ and $Y=N_{G}%
(w)\cup\{w\}$ have $X\cap Y=\emptyset$. But $\tau_{G}(X)$ contains the neighborhood
circuit $\zeta_{G}(v)$, and $\tau_{G}(Y)$ contains $\zeta_{G}(w)$, so condition 3 of
Theorem \ref{expdegree} is satisfied.
\end{proof}

The converse of the corollary is false as there are large graphs of diameter $\leq2$, but according to Corollary \ref{lowdeg} there is no graph with $n\geq7$ and $\kappa(M[IAS(G)])=n$.

\section{Theorem \ref{vconnect} and Corollary \ref{lowdeg}}
\label{sec:char_max_connect}

In this section we prove Theorem \ref{vconnect} by proving the following five statements for $M = M[IAS(G)]$. These statements differ slightly from those of Theorem \ref{vconnect}, as they have not been restricted so as to be mutually exclusive.
\begin{enumerate}

\item If $n\leq3$ and $G$ is connected, then $\kappa(M)=n$.

\item If $G$ is disconnected, then $\kappa(M)=1$.

\item If $G$ is connected and not prime, then $\kappa(M)=3$.

\item If $G$ is prime and $\kappa(M)<n$, then $\kappa(M)$ is odd and $\kappa(M) \geq 5$.

\item If $n\geq 4$ and $\kappa(M)=n$, then $G$ is locally equivalent to $C_{5}$ or $W_{5}$.
\end{enumerate}

Recall that the cut-rank function of $G$ is given by $c_G(X) = r(A[X,V-X])$ for $X\subseteq V(G)$, and it determines $\kappa(M)$ as in Corollary~\ref{vconn}. Statements 1, 2, 3 and 4 above follow from three obvious properties of the cut-rank function, which were observed by Bouchet \cite{Bconn}. For the reader's convenience we recall these obvious properties now.

Property 1. If $n\leq3$ and $G$ is connected, then every $X\subseteq V(G)$ has $c_G(X)=\min\{\left\vert X\right\vert ,\left\vert V(G)-X\right\vert \}$.

Property 2. $G$ is disconnected iff some $X \subseteq V(G)$ has $\emptyset \neq X \neq V(G)$ and $c_G(X)=0$. Equivalently, $G$ is connected iff $c_G(X) \geq 1$ for every $X \subseteq V(G)$ with $\emptyset \neq X \neq V(G)$.

Property 3. A partition $(V_1,V_2)$ of $V(G)$ provides a split of $G$ iff $\left\vert V_1\right\vert ,\left\vert V_2\right\vert \geq 2$ and $c_{G}(V_{1}) =c_{G}(V_{2}) \leq 1$.

The proof of statement 5 is considerably more difficult. We actually prove Corollary \ref{lowdeg} first, and then explain how to deduce statement 5 from the ``corollary.''

\subsection{Corollary \ref{lowdeg}}

\begin{definition}
\label{subtran}
Let $G$ be a graph, and let $\gamma$ be a subset of $W(G)$. Then
the sum (i.e., the symmetric difference)%
\[
S(\gamma)=\gamma+\sum_{\left\vert \gamma\cap\tau_{G}(v)\right\vert >1}\tau
_{G}(v)\text{.}%
\]
is the \emph{subtransversal associated to} $\gamma$.
\end{definition}

That is, $S(\gamma)$ is obtained from $\gamma$ in two steps: for each $v\in
V(G)$ with $\tau_{G}(v)\subseteq\gamma$, remove $\tau_{G}(v)$ from $\gamma$;
and for each $v\in V(G)$ with $\left\vert \gamma\cap\tau_{G}(v)\right\vert
=2$, replace the two elements of $\gamma\cap\tau_{G}(v)$ with the third
element of $\tau_{G}(v)$.

For ease of reference we state the following observation as a lemma.

\begin{lemma}
\label{dominate}Suppose $\gamma$ is an element of the cycle space of
$M[IAS(G)]$, i.e., $\gamma$ is a subset of $W(G)$ and the columns of $IAS(G)$
corresponding to $\gamma$ sum to $0$. If $\gamma$ is a union of vertex triples
then $S(\gamma)=\emptyset$; otherwise $S(\gamma)\neq\emptyset$ and
$S(\gamma)$\ is the disjoint union of some transverse circuits of $M[IAS(G)].$
\end{lemma}

\begin{proof}
The sum of the columns of $IAS(G)$ corresponding to elements of $\gamma$
equals the sum of columns corresponding to elements of $S(\gamma)$.
\end{proof}

\begin{lemma}
\label{tranrank}Let $G$ be a graph. Then the following properties of a subset $X\subseteq V(G)$ are equivalent.

1. The rank of $\tau_{G}(X)$ in $M[IAS(G)]$ is $\geq2\cdot\left\vert
X\right\vert $.

2. The rank of $\tau_{G}(X)$ in $M[IAS(G)]$ equals $2\cdot\left\vert
X\right\vert $.

3. No transverse circuit of $M[IAS(G)]$ is contained in $\tau_{G}(X)$.
\end{lemma}

\begin{proof}
As $\tau_{G}(X)$ contains the
vertex triples $\tau_{G}(x)$ with $x\in X$, and these vertex triples are
pairwise disjoint dependent sets of $M[IAS(G)]$, $r(\tau_{G}(X))\leq\left\vert
\tau_{G}(X)\right\vert -\left\vert X\right\vert =2\cdot\left\vert X\right\vert
$. This explains the equivalence between properties 1 and 2.

Suppose property 3 fails, i.e., $\tau_{G}(X)$ contains a transverse circuit
$\zeta$. Let $\tau_{G}(X)=S_{1}\cup S_{2}\cup S_{3}$, where the $S_{i}$ are
pairwise disjoint subtransversals of $W(G)$ and $\zeta\subseteq S_{1}$. As
vertex triples are dependent in $M[IAS(G)]$, $S_{3}$ is contained in the
closure of $S_{1}\cup S_{2}$. It follows that%
\[
r(\tau_{G}(X))=r(S_{1}\cup S_{2})\leq r(S_{1})+r(S_{2})\leq-1+\left\vert
S_{1}\right\vert +\left\vert S_{2}\right\vert \leq-1+2\cdot\left\vert
X\right\vert ,
\]
so property 1 also fails.

Suppose now that property 1 fails. Let $\tau_{G}(X)=S_{1}\cup S_{2}\cup S_{3}%
$, where the $S_{i}$ are pairwise disjoint subtransversals of $W(G)$. Again,
the fact that vertex triples are dependent in $M[IAS(G)]$ implies that
$r(\tau_{G}(X))=r(S_{1}\cup S_{2})<2\cdot\left\vert X\right\vert =\left\vert
S_{1}\cup S_{2}\right\vert $. Consequently, $S_{1}\cup S_{2}$ contains some circuit
$\gamma$ of $M[IAS(G)]$. Lemma \ref{dominate} tells us that there is a
transverse circuit $\zeta$ with $\zeta\subseteq S(\gamma)$.
\end{proof}

Here is an easy consequence of Lemma \ref{tranrank}.

\begin{corollary}
\label{tranrankcor}Let $G$ be a graph with $n$ vertices, and let
$X\subseteq V(G)$ be a subset with $\left\vert X\right\vert >\frac{n}{2}$.
Then $\tau_{G}(X)$ contains a transverse circuit.
\end{corollary}

\begin{proof}
As $2\cdot\left\vert X\right\vert >n$ and $n=r(W(G))$, $2\cdot\left\vert
X\right\vert >r(\tau_{G}(X))$.
\end{proof}

In general, if $\left\vert X\right\vert \leq\frac{n}{2}$ then $\tau_{G}(X)$
need not contain a transverse circuit. However, in Theorem \ref{halfcirc} we
prove that if $n\geq7$ then $M[IAS(G)]$ has some transverse circuit of size
$\leq\frac{n}{2}$. A couple of preliminary results will be useful.

\begin{lemma}
\label{halfclem}Let $G$ be a graph, and let $q$ be the size of
the smallest transverse circuit(s) in $M[IAS(G)]$. Suppose $Q\subseteq V(G)$
and $\left\vert Q\right\vert =q$. Then one of the following holds:

\begin{enumerate}
\item There is no transverse circuit contained in $\tau_{G}(Q)$.

\item There is precisely one transverse circuit contained in $\tau_{G}(Q)$.

\item There are precisely three transverse circuits contained in $\tau_{G}%
(Q)$, and they partition $\tau_{G}(Q)$.
\end{enumerate}
\end{lemma}

\begin{proof}
Suppose $\tau_{G}(Q)$ contains two transverse circuits $\zeta_{1}\neq\zeta
_{2}$. As $\zeta_{1}\Delta\zeta_{2}$ is not a union of vertex triples, Lemma
\ref{dominate} tells us that $S(\zeta_{1}\Delta\zeta_{2})$ contains a
transverse circuit, $\zeta_{3}$. If $\zeta_{1}\cap\zeta_{2}\not =\emptyset$
then $\left\vert S(\zeta_{1}\Delta\zeta_{2})\right\vert <q$, contradicting the
minimality of $q$. Consequently $\zeta_{1}\cap\zeta_{2}=\emptyset$ and
$S\left(  \zeta_{1}\Delta\zeta_{2}\right)  =\zeta_{3}.$
\end{proof}

\begin{corollary}
\label{halfclemcor}Let $G$ be a graph, and suppose $q\geq5$ is
the cardinality of the smallest transverse circuit(s) in $M[IAS(G)]$. Suppose
$Q\subseteq V(G)$, $\left\vert Q\right\vert =q$ and $\tau_{G}(Q)$ contains
three transverse circuits. Then for every $v\notin Q$, the only transverse
circuits contained in $\tau_{G}(Q\cup\{v\})$ are the ones contained in
$\tau_{G}(Q)$.
\end{corollary}

\begin{proof}
Suppose instead that $\tau_{G}(Q\cup\{v\})$ contains a transverse circuit
$\zeta$, which intersects $\tau_{G}(v)$. Then $\zeta$ also intersects at least
$q-1$ vertex triples in $\tau_{G}(Q)$; as $q-1\geq4$, $\zeta$ must share at
least two elements with a transverse circuit $\zeta^{\prime}$ contained in
$\tau_{G}(Q)$. Then Lemma \ref{dominate} tells us that $S(\zeta\Delta
\zeta^{\prime})$ contains a transverse circuit; but this is impossible as
$\left\vert S(\zeta\Delta\zeta^{\prime})\right\vert <q$.
\end{proof}

\begin{theorem}
\label{halfcirc}Let $G$ be a graph with $n\geq7$ vertices. Then
$M=M[IAS(G)]$ has a transverse circuit of size $\leq\frac{n}{2}$.
\end{theorem}

\begin{proof}
We have verified the theorem for $n=7$ and $n=9$ by exhaustion, using the matroid module \cite{sageMatroid} for SageMath \cite{sage} and the nauty package \cite{nauty}. For the rest of the proof, then, we assume that either $n=8$ or $n>9$.

Suppose the theorem fails for $G$, i.e., every transverse circuit of
$M$ is of size $>\frac{n}{2}$. For convenience we let $p=\left\lfloor
\frac{n}{2}\right\rfloor $, and we fix a subset $P\subseteq V(G)$ with
$\left\vert P\right\vert =p$. If $v\notin P$ then Corollary \ref{tranrankcor}
and Lemma \ref{halfclem} tell us that $\tau_{G}(P\cup\{v\})$ contains either
precisely one or precisely three transverse circuits. Each such transverse
circuit meets every vertex triple in $\tau_{G}(P\cup\{v\})$; and if there are three such transverse circuits, they constitute a partition of $\tau_{G}(P\cup\{v\})$.

Suppose there is a vertex $v_{1}\notin P$ such that $\tau_{G}(P\cup\{v_{1}\})$
contains three transverse circuits. If $v_{2}\notin P\cup\{v_{1}\}$ then
Corollary \ref{tranrankcor} tells us that $\tau_{G}(P\cup\{v_{2}\})$ contains
a transverse circuit. As $p\geq4$, however, Corollary \ref{halfclemcor} tells
that $\tau_{G}(P\cup\{v_{2}\})$ does not contain a transverse circuit. We
conclude by contradiction that there is no vertex $v_{1}\notin P$ such that
$\tau_{G}(P\cup\{v_{1}\})$ contains three transverse circuits. That is, for
each $v\notin P$ there is a unique transverse circuit $\zeta_{v}\subseteq
\tau_{G}(P\cup\{v\})$.

The rest of the proof is split into three cases.

Case 1: $n$ is even. If $v\notin P$ and $w\in\tau_{G}(v)$ is not included in
$\zeta_{v}$, then Lemma \ref{dominate} tells us that $w$ is not included in
any circuit $\gamma\subseteq\tau_{G}(P)\cup\{w\}$. That is, the rank of
$\tau_{G}(P)\cup\{w\}$ is $1+r(\tau_{G}(P))$. But $1+r(\tau_{G}(P))=1+2p=1+n$,
an impossibility as the rank of $M$ is only $n$.

Case 2: $n\geq13$ and odd. Let $V(G)=\{v_{1},\ldots,v_{2p+1}\}$, with $P=\{v_{1}%
,\ldots,v_{p}\}$, and for $i>p$ let $\zeta_{i}=\zeta_{v_{i}}$. Notice that if
$i,j>p$, $i\neq j$ and $\left\vert \zeta_{i}\cap\zeta_{j}\right\vert >1$ then
$\zeta_{i}\Delta\zeta_{j}$ meets no more than $p$ vertex triples, including
$\tau_{G}(v_{i})$ and $\tau_{G}(v_{j})$. Lemma \ref{dominate} then tells us
that $S(\zeta_{i}\Delta\zeta_{j})$ contains a transverse circuit of size $\leq
p$, a contradiction; we conclude that $\left\vert \zeta_{i}\cap\zeta
_{j}\right\vert \leq1$ whenever $i\neq j$. As each of $\zeta_{p+1}%
,\ldots,\zeta_{2p+1}$ includes $p$ elements of $\tau_{G}(P)$, the fact that
$\left\vert \zeta_{i}\cap\zeta_{j}\right\vert \leq1$ $\forall i\neq j$ implies
that
\[
\left\vert
{\displaystyle\bigcup\limits_{i=p+1}^{2p+1}}
(\tau_{G}(P)\cap\zeta_{i})\right\vert \geq\sum_{i=1}^{p+1}(p+1-i)=\frac
{p(p+1)}{2}\text{.}%
\]

As $\left\vert \tau_{G}(P)\right\vert =3p$, we conclude that $\frac{p+1}%
{2}\leq3$; this is impossible as $n\geq13$.

Case 3: $n=11$. The argument of case 2 applies, except for the last sentence.
We have $V(G)-P=\{v_{6},\ldots,v_{11}\}$, and for each $i\in\{6,7,8,9,10,11\}$%
\ we have a unique transverse circuit $\zeta_{i}\subseteq\tau_{G}(P\cup
\{v_{i}\})$. We use $\kappa_{i}$ to denote the element of $\zeta_{i}\cap
\tau_{G}(v_{i})$. If any one element of $\tau_{G}(P)$ is included in as many
as three of $\zeta_{6},\ldots,\zeta_{11}$ then counting elements as in case 2 tells us that $\tau
_{G}(P)\cap(\zeta_{6}\cup \cdots \cup\zeta_{11})$ has at least $5+4+4+2+1 = 16 > 15 = 3p = \left\vert \tau_{G}(P)\right\vert$ elements,
an impossibility. Consequently, no element of $\tau_{G}(P)$ is included in more
than two of the circuits $\zeta_{6},\ldots,\zeta_{11}$. As each $\zeta_{i}$
includes 5 elements of $\tau_{G}(P)$, and $\tau_{G}(P)$ has only 15 elements,
we conclude that each element of $\tau_{G}(P)$ appears in precisely two of
$\zeta_{6},\ldots,\zeta_{11}$.

Consider the total of the six column-sums corresponding to $\zeta_{6},\ldots,\zeta_{11}$. On the one hand, each column corresponding to an element of $\tau_{G}(P)$ appears twice; as we are working over $GF(2)$, the total equals the sum of the columns of $IAS(G)$ corresponding to $\kappa_{6},\ldots,\kappa_{11}$. On the other hand, each $\zeta_{i}$ is a circuit, so the total is $0$. As there is no transverse circuit of size $<6$, the subtransversal $\{\kappa_{6},\ldots,\kappa_{11}\}$ is a transverse circuit.

Recall that each of the transverse circuits $\zeta_{6},\ldots,\zeta_{11}$ intersects each of the vertex triples $\tau_{G}(v_{1}),\ldots,\tau_{G}(v_{5})$ precisely once, each element of one of these vertex triples appears in precisely two of $\zeta_{6},\ldots,\zeta_{11}$, and $\left\vert
\zeta_{i}\cap\zeta_{j}\right\vert \leq1$ $\forall i\neq j$. We claim that these constraints make it possible to index $v_1,\ldots,v_{11}$ and label the elements of $\tau_G(P)$ in such a way that for $1\leq k\leq5$ the vertex triple $\tau_G(v_k)$ is $\{\kappa_{k},\lambda_{k},\mu_{k}\}$ and the table below is correct, in the sense that the two
numbers in the table location corresponding to each element of $\tau_{G}(P)$ provide the two indices $i,j$ such that $\zeta_{i}\cap \zeta_{j}$ includes that element. 

\[%
\begin{tabular}
[c]{|c|c|c|c|c|c|}\hline
$\kappa$ & 6,7 & 6,8 & 8,10 & 7,9 & 9,10\\\hline
$\lambda$ & 8,9 & 7,10 & 6,9 & 6,10 & 7,8\\\hline
$\mu$ & 10,11 & 9,11 & 7,11 & 8,11 & 6,11\\\hline
$k$ & 1 & 2 & 3 & 4 & 5\\\hline
\end{tabular}
\
\]

We proceed to verify the claim. Without loss of generality we may presume that $v_{1},\ldots,v_{11}$ have been indexed in such a way that $\tau_G(v_1)$ contains both $\zeta_{6}\cap\zeta_{7}$ and $\zeta_{8}\cap\zeta_{9}$, while $\tau_G(v_2)$ contains $\zeta_{6}\cap\zeta_{8}$. The third element of $\tau_G(v_1)$ must be the one element of $\zeta_{10}\cap\zeta_{11}$, so the elements of $\tau_{G}(v_{1})$ may be denoted $\kappa_{1},\lambda_{1},\mu_{1}$ in such a way that the $k=1$ column of the table above correctly records the appearances of elements of $\tau_G(v_1)$ in $\zeta_{6},\ldots,\zeta_{11}$. As $\tau_G(v_2)$ contains $\zeta_{6}\cap\zeta_{8}$ it is not possible that $\tau_G(v_2)$ also contains $\zeta_{7}\cap\zeta_{9}$; for if it did, the third element of $\tau_G(v_2)$ would be the element of $\zeta_{10}\cap\zeta_{11}$, which appears in $\tau_G(v_1)$. Interchanging indices of $v_{10}$ and $v_{11}$ if necessary, we may presume that $\tau_G(v_2)$ contains $\zeta_{7}\cap\zeta_{10}$. Then we may name the elements of $\tau_G(v_2)$ in such a way that the $k=2$ column of the table above is correct. We may presume that $\tau_G(v_3)$ contains $\zeta_{6}\cap\zeta_{9}$. There is then only one value of $i$ for which $\tau_G(v_3)$ could possibly contain $\zeta_{i}\cap\zeta_{10}$: $i\in\{6,9\}$ is impossible as we have already identified an element of $\tau_G(v_3)$ in $\zeta_{6}\cap\zeta_{9}$, and $i\in\{7,11\}$ is impossible because $\tau_G(v_1)$ contains $\zeta_{10}\cap\zeta_{11}$ and $\tau_G(v_2)$ contains $\zeta_{7}\cap\zeta_{10}$. We conclude that $\tau_G(v_3)$ contains $\zeta_{8}\cap\zeta_{10}$. The third element of $\tau_G(v_3)$ must then be shared with $\zeta_{7}\cap\zeta_{11}$, and we can name the elements of $\tau_G(v_3)$ so that the $k=3$ column of the table is correct. Interchanging the indices of $v_4$ and $v_5$ if necessary, we may presume that $\tau_G(v_4)$ contains $\zeta_{6}\cap\zeta_{10}$ and $\tau_G(v_5)$ contains $\zeta_{6}\cap\zeta_{11}$. Then there is only one value of $i$ for which $\tau_G(v_4)$ could possibly contain $\zeta_{i}\cap\zeta_{11}$: $i\in\{7,9,10\}$ is impossible because of the elements already assigned to $\tau_G(v_1)$, $\tau_G(v_2)$ and $\tau_G(v_3)$; and $i=6$ is impossible because $\tau_G(v_4)$ contains an element of $\zeta_{6}\cap\zeta_{10}$. It follows that $\tau_G(v_4)$ contains $\zeta_{8}\cap\zeta_{11}$. The third element of $\tau_G(v_4)$ must then be the element of  $\zeta_{7}\cap\zeta_{9}$, and we may name the elements of $\tau_G(v_4)$ so that the $k=4$ column of the table above is correct. The $k=5$ column of the table is then forced, except for the assignment of $\kappa,\lambda,\mu$ labels to the elements of $\tau_G(v_5)$.

We now claim that the subtransversal $\{\kappa_{1},\ldots,\kappa_{10}\}$  is an independent set of $M$. Suppose the claim is false; then there is a subset $\zeta \subseteq \{\kappa_{1},\ldots,\kappa_{10}\}$ such that the columns of $IAS(G)$ corresponding to elements of $\zeta$ sum to $0$. Let $Y=\{y \geq 6 \mid \kappa_y \in \zeta \}$, let $Z=\{\zeta\} \cup \{\zeta_y \mid y \in Y \}$, and let $\gamma=\{x \in \tau_G(P) \mid x \text{ is included in an odd number of elements of }Z \}$. Consider the sum of columns of $IAS(G)$ corresponding to elements of elements of $Z$. As we are working over $GF(2)$, any column that appears an even number of times in the sum contributes $0$. Each column corresponding to a $\kappa_y$ with $y \in Y$ appears twice in the sum, once in $\zeta$ and once in $\zeta_y$, so these columns contribute $0$. It follows that the sum of columns of $IAS(G)$ corresponding to elements of elements of $Z$ is the same as the sum of columns of $IAS(G)$ corresponding to elements of $\gamma$. This sum must be $0$, as the individual column-sums corresponding to $\zeta$ and the $\zeta_y$ are all $0$. We conclude that the sum of columns of $IAS(G)$ corresponding to elements of $\gamma$ is $0$. Let $S(\gamma)$ be the subtransversal associated to $\gamma$, as in Definition~\ref{subtran}. Then $S(\gamma) \subseteq \tau_G(P)$, so $S(\gamma)$ does not contain any transverse circuit of $M$. According to Lemma~\ref{dominate}, it follows that $S(\gamma)=\emptyset$. That is, for each $i \in \{1,\ldots,5\}$ the intersection $\gamma \cap \tau_G(v_i)$ is either $\emptyset$ or $\tau_G(v_i)$. Inspecting the table above, though, we see that no nonempty set $Y \subseteq \{6,\ldots,10\}$ respects the requirement that every $i \in \{1,\ldots,5\}$ has $\gamma \cap \tau_G(v_i) \in \{\emptyset,\tau_G(v_i)\}$. For instance, if $9 \in Y$ then $\mu_2 \in \gamma$, so $\gamma \cap \tau_G(v_2)=\tau_G(v_2)$, so $Y$ must include precisely one of $6,8$ and precisely one of $7,10$; but $Y=\{6,7,9\}$ and $Y=\{6,9,10\}$ both violate $\gamma \cap \tau_G(v_3) \in \{\emptyset,\tau_G(v_3)\}$, and $Y=\{7,8,9\}$ and $Y=\{8,9,10\}$ both violate $\gamma \cap \tau_G(v_1) \in \{\emptyset,\tau_G(v_1)\}$. Of course $Y=\emptyset$ is impossible too, as $\{\kappa_{1},\ldots,\kappa_{5}\}$ is independent. These observations verify the claim.

As $\{\kappa_{6},\ldots,\kappa_{11}\}$ is a circuit of $M$, the claim implies that $\{\kappa_{1},\ldots,\kappa_{10}\}$ is an independent subtransversal whose closure includes $\kappa_{11}$. As explained in Section
4 of the first paper in this series~\cite{BT1}, after replacing $G$ with a locally
equivalent simple graph (if necessary) we may presume that $\kappa_{i}%
=\phi_{G}(v_{i})$ for $1\leq i\leq10$, and $\kappa_{11}=\chi_{G}(v_{11})$. (N.b. Our definition of local equivalence includes loop complementations, so we lose no generality when we presume that $G$ is simple.) The
fact that $\{\kappa_{6},\ldots,\kappa_{11}\}$ is a transverse circuit implies
that $N_{G}(v_{11})=\{v_{6},\ldots,v_{10}\}$. We have no information about
adjacencies among the vertices of $N_{G}(v_{11})$, but we can determine the
other adjacencies in $G$ as follows.

Let $S=\{\kappa_1,\kappa_2,\kappa_3,\kappa_4,\kappa_5,\lambda_1,\lambda_2,\lambda_3,\lambda_4,\lambda_5\}$. The closure of $S$ in $M$ contains an element $x \notin S$ if and only if $x$ is included in a circuit of $M$ whose other elements are contained in the closure of $S$. It follows that the closure of $S$ contains $\mu_1,\ldots,\mu_5$ (which are included in vertex triples in $\tau_G(P)$) and $\kappa_6,\ldots,\kappa_{11}$ (which are included in $\zeta_6,\ldots,\zeta_{11}$). We conclude that the submatroid $M\mid(\tau_{G}(P)\cup\{\kappa_{6},\ldots,\kappa_{11}\})$ is spanned by $S$, so its rank is $\leq 10$.

As $B=\{\kappa_{1}=\phi_{G}(v_{1}),\ldots,\kappa_{10}=\phi_{G}(v_{10})\}$ is an independent set of rank $10$, $B$ must be a basis of the submatroid $M\mid(\tau_{G}(P)\cup\{\kappa_{6},\ldots,\kappa_{11}\})$. We can find fundamental circuits with respect to $B$ by searching for sums (symmetric differences) of the $\zeta_i$ which contain only one element not included in $B$. Here are some of these fundamental circuits:

\setlength{\tabcolsep}{3pt}
\begin{tabular}{lllll}
For $\mu_{1}$: & $\{\mu_{1},\phi_{G}(v_{2}),\phi_{G}(v_{4}),\phi_{G}(v_{6}),\phi_{G}(v_{7}),\phi_{G}(v_{10})\}$ & $=$ & $\zeta_6 \Delta \zeta_7 \Delta \zeta_{10}$.\\

For $\mu_{2}$: & $\{\mu_{2},\phi_{G}(v_{1}),\phi_{G}(v_{3}),\phi_{G}(v_{6}),\phi_{G}(v_{8}),\phi_{G}(v_{9})\}$ & $=$ & $\zeta_6 \Delta \zeta_8 \Delta \zeta_{9}$.\\

For $\mu_{3}$: & $\{\mu_{3},\phi_{G}(v_{2}),\phi_{G}(v_{5}),\phi_{G}(v_{7}),\phi_{G}(v_{8}),\phi_{G}(v_{10})\}$ & $=$ & $\zeta_7 \Delta \zeta_8 \Delta \zeta_{10}$.\\

For $\mu_{4}$: & $\{\mu_{4},\phi_{G}(v_{1}),\phi_{G}(v_{5}),\phi_{G}(v_{7}),\phi_{G}(v_{8}),\phi_{G}(v_{9})\}$ & $=$ & $\zeta_7 \Delta \zeta_8 \Delta \zeta_{9}$.\\

For $\mu_{5}$: & $\{\mu_{5},\phi_{G}(v_{3}),\phi_{G}(v_{4}),\phi_{G}(v_{6}),\phi_{G}(v_{9}),\phi_{G}(v_{10})\}$ & $=$ & $\zeta_6 \Delta \zeta_9 \Delta \zeta_{10}$.\\

For $\kappa_{11}$: & $\{\kappa_{11},\phi_{G}(v_{6}),\phi_{G}(v_{7}),\phi_{G}(v_{8}),\phi_{G}(v_{9}),\phi_{G}(v_{10})\}$ & $=$ & $\zeta_6 \Delta \zeta_7 \Delta \zeta_8 \Delta \zeta_9 \Delta \zeta_{10}$.

\end{tabular}

For $i \in \{1,\ldots,5\}$, the fundamental circuit of $\lambda_i$ is equal to the symmetric difference of $\tau_{G}(v_i)$ and the fundamental circuit of $\mu_i$.

\begin{figure}[tb]%
\centering
\includegraphics[scale=0.8]{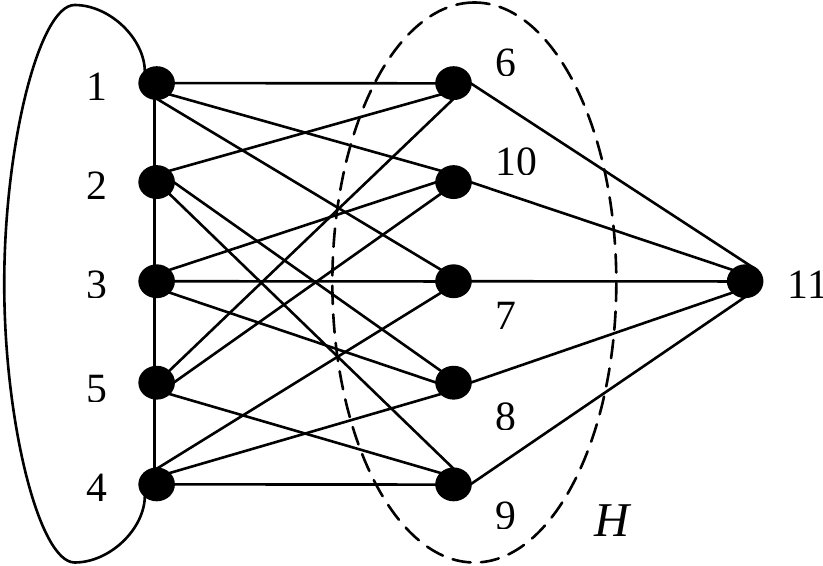}%
\caption{The case $n=11$, with unspecified edges in the induced subgraph $H$.}%
\label{isomch3f2}%
\end{figure}

The fundamental circuits indicate that $G$ is of the form pictured in Figure
\ref{isomch3f2} (with each vertex $v_i$ abbreviated by $i$), but we do not know which edges appear in the induced subgraph $H$ of $G$ with $V(H)=\{v_{6},v_{7},v_{8},v_{9},v_{10}\}$. Notice that, if we like, we may perform a local
complementation at $v_{11}$ without changing $M[IAS(G)]$ (up to isomorphism)
and without changing any information about $G$ that has been mentioned above.
The effect of a local complementation at $v_{11}$ is to complement all edges
in $H$, so we may assume that $\left\vert E(H)\right\vert \leq5$ without loss
of generality.

The argument for $n=11$ ends with the claim that no matter which edges appear in $H$, $M[IAS(G)]$ has a transverse circuit of size $\leq5$. The claim is justified by considering different possible configurations of edges in $H$, as follows.

If $H$ has an isolated vertex, then that vertex is of degree 4 in $G$, so its
neighborhood circuit is of size 5.

If $H$ is disconnected but has no isolated vertex then $H$ has a connected
component of size 2 and a connected component of size 3. There are two
distinct configurations of this type. For instance, if $\{v_{7},v_{8}\}$ forms a
connected component of $H$ then $\{\phi_{G}(v_{1})$, $\phi_{G}(v_{2})$,
$\psi_{G}(v_{7})$, $\psi_{G}(v_{8})\}$ is a transverse circuit. On the other
hand, if $\{v_{7},v_{9}\}$ forms a connected component of $H$ then so does
$\{v_{6},v_{8},v_{10}\}$. No matter which edges appear in the larger
component, there will be a transverse circuit of the form $\{\phi_{G}(v_{4})$,
$\rho_{G}(v_{6})$, $\sigma_{G}(v_{8})$, $\upsilon_{G}(v_{10})$, $\phi
_{G}(v_{11})\}$ with $\rho,\sigma,\upsilon\in\{\chi,\psi\}$.

If $H$ is connected then as $\left\vert E(H)\right\vert \leq5$, $H$ has a
vertex of degree 2 or a vertex of degree 4. For instance, if $v_{6}$ is of
degree 4 in $H$ then $\{\phi_{G}(v_{1})$, $\phi_{G}(v_{2})$, $\phi_{G}(v_{5}%
)$, $\psi_{G}(v_{6})$, $\chi_{G}(v_{11})\}$ is a transverse circuit. There are
three different configurations of degree-2 vertices. For instance, if
$N_{H}(v_{8})=\{v_{7}$, $v_{9}\}$ then $\{\chi_{G}(v_{1})$, $\phi_{G}(v_{4})$,
$\chi_{G}(v_{5})$, $\chi_{G}(v_{8})$, $\phi_{G}(v_{11})\}$ is a transverse
circuit; if $N_{H}(v_{8})=\{v_{9}$, $v_{10}\}$ then $\{\phi_{G}(v_{2})$,
$\chi_{G}(v_{5})$, $\phi_{G}(v_{6})$, $\chi_{G}(v_{8})$, $\phi_{G}(v_{11})\}$
is a transverse circuit; and if $N_{H}(v_{8})=\{v_{6}$, $v_{10}\}$ then
$\{\phi_{G}(v_{2})$, $\chi_{G}(v_{5})$, $\chi_{G}(v_{8})$, $\phi_{G}(v_{9})$,
$\phi_{G}(v_{11})\}$ is a transverse circuit.
\end{proof}

\subsection{Case 5 of Theorem \ref{vconnect}}

With Theorem \ref{degree} and Corollary \ref{lowdeg} proven, we know that if $n\geq5$ and $\kappa(M[IAS(G)])$ $=n$, then $n<7$. To complete the proof of Theorem \ref{vconnect}, we must show that such a $G$ is locally equivalent to either the cycle graph $C_{5}$ or the wheel graph $W_{5}$.

As $\kappa(M[IAS(G)])=n>3$, case 3 of Theorem \ref{vconnect} tells us that $G$ is prime. If $n=5$ we refer to Bouchet \cite{Bec}, who showed that every prime 5-vertex graph is locally equivalent to $C_{5}$. Suppose $n=6$. If $G$ is a circle graph then as discussed in \cite{BT2}, $G$ has a transverse circuit of size $\leq3$; but then Theorem \ref{degree} contradicts the hypothesis that $\kappa(M[IAS(G)])=n$. Consequently $G$ is not a circle graph. According to Bouchet's circle graph obstructions theorem \cite{Bco}, every non-circle graph with $n=6$ is locally equivalent to $W_{5}$. 

\section{A characterization of circle graphs}
\label{circle}

In this section we briefly discuss a way to use the ideas of this paper to characterize circle graphs. We refer to the second paper in the series~\cite{BT2} for definitions, and for the following.

\begin{theorem}
(\cite{BT2}) Let $G$ be the interlacement graph of an Euler system of a 4-regular graph $F$. If $F$ has a circuit of size $q$, then $G$ has a transverse circuit of size $\leq q$.
\end{theorem}

Proposition \ref{smallcirc} immediately implies the following.

\begin{corollary}
\label{circgirth}Let $G$ be the interlacement graph of an Euler system of a
4-regular graph $F$. If the girth of $F$ is $g(F)\leq\frac{n}{2}$, then
$\kappa(M[IAS(G)])\leq2g(F)-1$.
\end{corollary}

With a little more work we obtain an upper bound on the vertical connectivity of the isotropic matroid of a circle graph.

\begin{corollary}
\label{circgirth2}Let $G$ be a circle graph. Then $\kappa(M[IAS(G)])\leq \max \{5,n-3\}$.
\end{corollary}

\begin{proof}
Every graph with $n\leq5$ is a circle graph, and satisfies $\kappa(M[IAS(G)])\leq n \leq5$. If $n\geq6$ and $G$ is associated with a 4-regular graph $F$ with $g(F)\leq3$, then Corollary~\ref{circgirth} tells us that $\kappa(M[IAS(G)])\leq 5$.

According to the tables of Meringer \cite{Me}, for $6\leq n\leq9$ the only 4-regular graph of girth $>3$ is $K_{4,4}$. It turns out that circle graphs associated with $K_{4,4}$ have $\kappa(M[IAS(G)])=5$. To see why, recall the observation of \cite{BT2} that
\[
\noindent%
\begin{pmatrix}
0 & 1 & 0 & 0 & 1 & 0 & 0 & 1\\
1 & 0 & 0 & 0 & 1 & 1 & 0 & 0\\
0 & 0 & 0 & 1 & 1 & 1 & 0 & 0\\
0 & 0 & 1 & 0 & 1 & 0 & 0 & 1\\
1 & 1 & 1 & 1 & 0 & 0 & 1 & 0\\
0 & 1 & 1 & 0 & 0 & 0 & 1 & 0\\
0 & 0 & 0 & 0 & 1 & 1 & 0 & 1\\
1 & 0 & 0 & 1 & 0 & 0 & 1 & 0
\end{pmatrix}
\]
is the adjacency matrix of an interlacement graph $\mathcal{I}(C)$ of an Euler circuit $C$ of $K_{4,4}$. Note that if $X$ includes the vertices corresponding to the last four rows and columns then
\[
c_{\mathcal{I}(C)}(X)=r(A[X,V(\mathcal{I}(C))-X])=
r\left(\begin{pmatrix}
1 & 0 & 0 & 1\\
1 & 1 & 0 & 0\\
1 & 1 & 0 & 0\\
1 & 0 & 0 & 1
\end{pmatrix}
\right)=2 \text{,}
\]
so $\kappa(M[IAS(\mathcal{I}(C))])\leq5$. As $\mathcal{I}(C)$ is prime, it follows that $\kappa(M[IAS(\mathcal{I}(C))])=5$. We conclude that every circle graph with $n\leq9$ has $\kappa(M[IAS(G)])\leq5$.

If $n\geq10$ and $G$ is the interlacement graph of an Euler system of a 4-regular graph of girth 4, then Corollary~\ref{circgirth} tells us that $\kappa(M[IAS(G)])\leq7$. As $7\leq n-3$, the inequality asserted in the present corollary is satisfied.

It remains to consider circle graphs with $n\geq10$ that are associated with 4-regular graphs of girth $>4$. The Moore bounds for the order of a regular graph of given girth are well known; see~\cite[Chapter 23]{B} for a discussion. The Moore bounds tell us that the order of a 4-regular graph of girth $g>4$ satisfies the following inequality.
\[
n\geq
\begin{cases}
1+4 \cdot (1+3+\cdots +3^{(g-3)/2}) & \text{if }g\text{ is odd}\\
1+4 \cdot (1+3+\cdots +3^{(g-4)/2})+3^{(g-2)/2} & \text{if }g\text{ is even}
\end{cases}.
\]
If $g\geq5$ is odd then we deduce that 
\[
n\geq 1+4+12((g-3)/2)=5+6g-18=2g+4g-13\geq 2g+7 \text{,}
\]
and if $g\geq6$ is even then we deduce that
\[
n\geq 1+4+12((g-4)/2)+9=5+6g-15=2g+4g-10\geq 2g+14 \text{.}
\]
We require only the fact that $g>4$ implies $n> 2g+2$, as this allows us to apply Corollary~\ref{circgirth} and conclude that $\kappa(M[IAS(G)])\leq 2g-1 < n-3$.
\end{proof}

We should mention that the bound $\kappa(M[IAS(G)])\leq n-3$ of Corollary~\ref{circgirth2} is sharp for $n=10$, although it is certainly not sharp for $n>10$. Consulting Meringer's tables \cite{Me}, we see that there are two 4-regular graphs of order 10 and girth $>3$. One of these graphs is obtained from $C_{5}$ by doubling every vertex and quadrupling every edge, and the other is obtained from $K_{5,5}$ by removing a perfect matching. Computations performed using the matroid module of SageMath \cite{sageMatroid, sage} indicate that circle graphs associated with these two 4-regular graphs have $\kappa(M[IAS(G)])=7$.

It is not hard to see that $\kappa(M[IAS(W_{5})])=6$ and $\kappa(M[IAS(W_{7})])=7$, so according to Corollary~\ref{circgirth2} the fact that $W_{5}$ and $W_{7}$ are not circle graphs is detected by the high vertical connectivity of their isotropic matroids. Bouchet gave three forbidden vertex-minors for circle graphs \cite{Bco}: $W_{5}$ and $W_{7}$ are two of them, and the third is a bipartite graph denoted $BW_{3}$. The vertical connectivity of the isotropic matroid of $BW_{3}$ is only 5, as $BW_{3}$ has vertices of degree 2, so in order to exclude $BW_{3}$ we must use another property of circle graphs.

\begin{theorem}
The family $\mathcal{C}$ of circle graphs is determined by these three properties.
\begin{itemize}
\item $\mathcal{C}$ is closed under vertex-minors.
\item If $G\in \mathcal{C}$ then all transverse matroids of $G$ are cographic.
\item If $G\in \mathcal{C}$ then $\kappa(M[IAS(G)])\leq \max \{5,\left\vert V(G) \right\vert -3\}$.
\end{itemize}
\end{theorem}

\subsection*{Acknowledgements}
We thank James Oxley for his useful comments regarding vertical connectivity. We are also grateful to an anonymous reader, whose careful reading and constructive comments resulted in several improvements in the paper. R.B.\ is a postdoctoral fellow of the Research Foundation -- Flanders (FWO).


\begin{thebibliography}{99}

\bibitem {A}L. Allys, \emph{Minimally 3-connected isotropic systems}, Combinatorica \textbf{14} (1994), 247--262. \doi{10.1007/BF01212973}

\bibitem {B} N. Biggs, \emph{Algebraic Graph Theory}, Second Edition, Cambridge Univ. Press, Cambridge, 1993. \doi{10.1017/CBO9780511608704}

\bibitem {Bi2}A. Bouchet, \emph{Graphic presentation of isotropic systems}, J. Combin. Theory Ser. B \textbf{45} (1988), 58--76. \doi{10.1016/0095-8956(88)90055-X}

\bibitem {Bdm}A. Bouchet, \emph{Greedy algorithm and symmetric matroids}, Math. Programming \textbf{38} (1987), 147--159. \doi{10.1007/BF02604639}

\bibitem {Bi1}A. Bouchet, \emph{Isotropic systems}, European J. Combin. \textbf{8} (1987), 231--244. \doi{10.1016/S0195-6698(87)80027-6}

\bibitem {Bec}A. Bouchet, \emph{Reducing prime graphs and recognizing circle graphs}, Combinatorica \textbf{7} (1987), 243--254. \doi{10.1007/BF02579301}

\bibitem {Bconn}A. Bouchet, \emph{Connectivity of isotropic systems}, Ann. N. Y. Acad. Sci. \textbf{555} (1989), 81--93. \doi{10.1111/j.1749-6632.1989.tb22439.x}

\bibitem {Bco}A. Bouchet, \emph{Circle graph obstructions}, J. Combin. Theory Ser. B \textbf{60} (1994), 107--144. \doi{10.1006/jctb.1994.1008}

\bibitem {BG}A. Bouchet and L. Ghier, \emph{Connectivity and }$\beta$\emph{-invariants of isotropic systems and 4-regular graphs}, Discrete Math. \textbf{161} (1996), 25--44. \doi{10.1016/0012-365X(95)00219-M}

\bibitem {BT1}R. Brijder and L. Traldi, \emph{Isotropic matroids I:
Multimatroids and neighborhoods}, Electron. J. Combin. \textbf{23}(4) (2016), \#P4.1. \url{http://www.combinatorics.org/ojs/index.php/eljc/article/view/v23i4p1}

\bibitem {BT2}R. Brijder and L. Traldi, \emph{Isotropic matroids II:
Circle graphs}, Electron. J. Combin. \textbf{23}(4) (2016), \#P4.2. \url{http://www.combinatorics.org/ojs/index.php/eljc/article/view/v23i4p2}

\bibitem {Cu}W. H. Cunningham, \emph{Decomposition of directed graphs}, SIAM J. Alg. Disc. Meth. \textbf{3} (1982), 214--228. \doi{10.1137/0603021}

\bibitem {nauty} B. D. McKay and A. Piperno, \emph{Practical graph isomorphism, II}, J. Symb. Comput. \textbf{60} (2014), 94--112. \doi{10.1016/j.jsc.2013.09.003}

\bibitem {Me}M. Meringer, \emph{Fast generation of regular graphs and construction of cages}, J. Graph Theory \textbf{30} (1999), 137--146. \doi{10.1002/(SICI)1097-0118(199902)30:2<137::AID-JGT7>3.0.CO;2-G} See also http://www.mathe2.uni-bayreuth.de /markus/reggraphs.html\#CRG

\bibitem {O}J. G. Oxley, \emph{Matroid Theory}, Second Edition, Oxford Univ. Press, Oxford, 2011. \doi{10.1093/acprof:oso/9780198566946.001.0001}

\bibitem {OSW} J. Oxley, C. Semple and G. Whittle, \emph{The structure of the 3-separations of 3-connected matroids}, J. Combin. Theory Ser. B \textbf{92} (2004), 257--293. \doi{10.1016/j.jctb.2004.03.006}

\bibitem {OSW2} J. Oxley, C. Semple and G. Whittle, \emph{The structure of the 3-separations of 3-connected matroids II}, European J. Combin. \textbf{28} (2007), 1239--1261. \doi{10.1016/j.ejc.2006.01.007}

\bibitem {sageMatroid}R. Pendavingh and S. van Zwam, \emph{Matroid theory module for SageMath}, retrieved 2015.

\bibitem {sage}The Sage Developers, \emph{{S}age {M}athematics {S}oftware}, 2015, \url{http://www.sagemath.org}.

\bibitem {Tnewnew}L. Traldi, \emph{Binary matroids and local complementation}, European J. Combin. \textbf{45} (2015), 21--40. \doi{10.1016/j.ejc.2014.10.001}

\end{thebibliography}
\end{document}